\documentclass[12pt,a4paper]{amsart}
\usepackage{amsfonts}
\usepackage[active]{srcltx}
\usepackage[notref,notcite]{showkeys}
\usepackage{enumerate}

\usepackage{amsmath,amssymb,xspace,amsthm}
\newtheorem{theorem}{Theorem}

\newtheorem{example}{Example}
\newtheorem{remark}[theorem]{Remark}
\newtheorem{lemma}[theorem]{Lemma}
\newtheorem{proposition}[theorem]{Proposition}
\newtheorem{corollary}[theorem]{Corollary}

\numberwithin{equation}{section}

\def\Vir{\mathcal{V}}

\def\Der{\operatorname{Der}}
\def\span{\operatorname{span}}

\newcommand{\C}{\ensuremath{\mathbb C}\xspace}

\renewcommand{\a}{\ensuremath{\alpha}}
\renewcommand{\b}{\ensuremath{\beta}}

\newcommand{\LL}{{\widetilde{\mathcal{V}}}}

\newcommand{\ord}{{\rm ord}}
\newcommand{\Aut}{{\rm Aut}}
\renewcommand{\H}{\mathfrak{H}}

\newcommand{\ma}{\ensuremath{\mathfrak{a}}}
\renewcommand{\aa}{{\widetilde{\mathfrak{a}}}}
\newcommand{\Z}{\ensuremath{\mathbb{Z}}\xspace}
\newcommand{\N}{\ensuremath{\mathbb{N}}\xspace}

\newcommand{\V}{\ensuremath{\mathcal{V}}\xspace}

\newcommand{\Ind}{\ensuremath{\operatorname{Ind}}\xspace}
\newcommand{\Soc}{\ensuremath{\operatorname{Soc}}\xspace}

\newcommand{\ad}{\operatorname{ad}\xspace}

\newcommand{\ann}{\operatorname{ann}\xspace}

\newcommand{\WW}{\widetilde{\mathcal{W}}}
\newcommand{\W}{\mathcal{W}}
\renewcommand{\phi}{\varphi}

\begin{document}
\title[Generalized oscillator representations]{Generalized oscillator representations of  the twisted Heisenberg-Virasoro algebra}
\author{Rencai L{\"u} and Kaiming Zhao}
\date{Aug.28, 2013}
\maketitle

\begin{abstract} In this paper, we first obtain a general result on
sufficient conditions for tensor product modules  to be simple over
an arbitrary Lie algebra. We  classify simple modules with a nice
property  over the infinite-dimensional Heisenberg algebra ${\H}$,
and then obtain a lot of simple modules over the twisted
Heisenberg-Virasoro algebra $\LL$ from generalized oscillator
representations of $\LL$  by extending these $\H$-modules. We give the necessary and sufficient conditions for
Whittaker modules over $\LL$ (in the more general setting) to be simple. We use
the ``shifting technique" to determine the necessary and sufficient
conditions for the tensor products of highest weight modules and
modules of intermediate series over $\LL$ to be simple.  At last we
 establish the ``embedding trick" to obtain a lot more  simple $\LL$-modules.
\end{abstract}


\vskip 5pt \noindent {\em 2000  Math. Subj. Class.:} 17B10, 17B20,
17B65, 17B66, 17B68

\vskip 10pt

\section{Introduction}
We denote by $\mathbb{Z}$, $\mathbb{Z}_+$, $\N$, and $\mathbb{C}$
the sets of  all integers, nonnegative integers, positive integers,
 and complex numbers, respectively. For a Lie algebra
$L$ we denote by $U(L)$ the universal enveloping algebra of $L$.

  The twisted Heisenberg-Virasoro algebra $\LL$ is
the universal central extension of the Lie algebra
$\{f(t)\frac{d}{dt}+g(t)| f,g \in \C[t,t^{-1}]\}$ of differential
operators of order at most one on the Laurent polynomial algebra
$\C[t,t^{-1}]$, see \cite{ACKP}. More precisely, the twisted
Heisenbeg-Virasoro algebra ${\LL}$ is a Lie algebra over $\C$ with
the basis
$$\{t^{n+1}\frac{d}{d t},t^n ,z_1, z_2,z_3 | n \in \Z\}$$
and subject to the Lie brackets given by
\begin{equation}[t^{n+1}\frac{d}{d t},t^{m+1}\frac{d}{d t}]=(m-n)t^{m+n+1}\frac{d}{d t}+\delta_{n,-m}\frac{n^3-n}{12}z_1,\end{equation}
\begin{equation}[t^{n+1}\frac{d}{d t},t^m]=mt^{m+n}+\delta_{n,-m}(n^2+n)z_2, \end{equation}
\begin{equation}[t^n,t^m]=n\delta_{n,-m}z_3,\end{equation}
\begin{equation}[\LL,z_1]=[\LL,z_2]=[\LL,z_3]=0. \end{equation}
The center of $\LL$ is spanned by $z_1, z_2, z_3, $ and $t^0$.  We
define $d_n=t^{n+1}\frac{d}{d t }$, $I_i=t^i$, and both symbols will
be used according to contexts. The Lie algebra $\LL$ has a natural
$\Z$-gradation with respect to $\ad(d_0)$:
\begin{equation}\LL_n=\C d_n+\C I_n,\forall n\in \Z\setminus\{0\},\end{equation}
\begin{equation}\LL_0=\C d_0+\C I_0+\C z_1+\C z_2+\C
z_3.\end{equation}

The Lie algebra $\LL$ has a Witt subalgebra
$\mathcal{W}=\C[t]\frac{d}{d t}$, a Virasoro subalgebra $\Vir$ with
basis $\{d_i,z_1|i\in \Z\}$, and a Heisenberg subalgebra $\H$ with
basis $\{I_i, z_3|i\in \Z\}$.

\
The twisted Heisenberg-Virasoro algebra $\LL$ has been studied by  Arbarello, De Concini,  Kac, and Procesi
in \cite{ACKP}, where a connection is established between the second cohomology of
certain moduli spaces of curves and the second cohomology of the Lie algebra of
differential operators of order at most one. They also proved that when
the central element of the Heisenberg subalgebra acts in a non-zero way, a simple
highest weight module for $\LL$ is isomorphic to the tensor product of a simple
module for the Virasoro algebra and a simple module for the infinite-dimensional
Heisenberg algebra. For a more general result, see Theorem 12.

To introduce our results in this paper, we will first recall and define some concepts.

For $\lambda\in\C, s,r\in \Z_+$, we define  the following subalgebras and quotient algebras of $\LL$:
\begin{equation}\W=\Der(\C[t])=\span \{d_i|i\ge
-1\},\,\,\,\,\widetilde{\W}=\C[t]\frac{d}{d
t}+\C[t],\end{equation}
\begin{equation}\mathfrak{a}=\span \{d_i |i\ge 0\},\,\,\,\,\widetilde{\mathfrak{a}}=\span \{d_i,I_i |i\ge 0\},\end{equation}

\begin{equation}\V^{(r)}=\span \{d_i|i\ge r\},\,\,\LL^{(r,s)}=\span \{t^{r+i},d_{s+i}|i\ge 0\},\end{equation}
\begin{equation}\mathfrak{a}_r=\mathfrak{a}/\V^{(r+1)},\,\,\,\aa_{r,s}=\aa/\LL^{(r,s)},\end{equation}
\begin{equation}\Vir [\lambda]=\C[t,t^{-1}](t-\lambda)\frac{d}{d t}+\C z,\end{equation}
\begin{equation}\LL[\lambda]=\span\{t^i,t^i(t-\lambda)\frac{d}{d t}, z_1,z_2,z_3|i\in \Z\}.\end{equation}

For any $\Z$-graded  Lie algebra $L=\oplus L_i$ (which can be any of the above algebras), denote by
$\mathcal{O}_{L}$ the category of all $L$-modules $V$ satisfying

 {\bf Condition A:} For any
$v\in V$, there exists a positive integer $n$ (depending on $v$)
such that $L_{i} v=0$ for all $i\ge n$.

Let $V$ be a module over a Lie algebra $L$. We say that   $V$ is
trivial if $L V=0$. Denote by $\Soc_{L}(V)$ the socle of the
$L$-module $V$, i.e., $\Soc_{L}(V)$ is the sum of the minimal
nonzero submodules of $V$.

Recall that a module $V$ over a Lie algebra $L$ is called {\em
locally finite} provided that any $v\in V$ belongs to a finite
dimensional $L$-submodule. The module is called {\em locally
nilpotent} provided that for any $v\in V$ there exists an
$n\in\mathbb{N}$ such that $a_1a_2\cdots a_n(v)=0$ for all
$a_1,a_2,\dots,a_n\in L$.

When $V$ is a simple module over $L$, where $L$ is one of
$\LL,\V,\WW,\W$, the following lemma gives some equivalent
conditions for Condition A.

\begin{lemma}\label{eq-condition} Let $L$ be one of $\LL,\V,\WW,\W$, and $V$ be a simple $L$-module. Denote $L^{(k)}=\sum_{i\ge k}L_i$.
Then the following conditions are equivalent:
\begin{enumerate}[$($a$)$]
\item\label{1.1} $V\in \mathcal{O}_{L}$.
\item\label{1.2} There exists some $0\ne v\in V$ and $s\in \N$ such
that $L^{(s)} v=0$.
\item\label{1.3} There exists $k\in \mathbb{N}$ such that $V$ is a locally finite
$L^{(k)}$-module.
\item\label{1.4} There exists $m\in \mathbb{N}$ such that $V$ is a locally nilpotent
$L^{(m)}$-module.
\end{enumerate}
\end{lemma}
\begin{proof} $(a)\Rightarrow (b)$ is trivial. And using PBW Theorem, we can
easily deduce $(b)\Rightarrow (a),(c),(d)$.\par
 $(c)\Rightarrow (b).$ Let $0\ne v_0\in V$. Then $W=U(L^{(k)})v_0$ is a
 finite dimensional $L^{(k)}$ module. Hence
 $L^{(k)}/\ann_{L^{(k)}}(W)$ is finite dimensional. By similar arguments as in Section
 3.3 in \cite{MZ}, we obtain $d_{k_1}\in \ann_{L^{(k)}}(W)$ for some $k_1>k$, hence $L^{(n)}\subset \ann_{L^{(k)}}(W)$ for
 some $n>k+k_1$. Thus $L^{(k)}/\ann_{L^{(k)}}(W)$ is solvable.
 Hence $L^{(k)}$ has a common eigenvector $v\in W$. However there exists some $s\in \N$ such that $L^{(s)}\subset [L^{(k)},L^{(k)}]$,
 which implies $L^{(s)}v=0$.

 $(d)\Rightarrow (b)$. Let $0\ne v\in V$. There exists some
 $n\in \N$ with $a_1a_2\cdots a_n v=0$ for all
$a_1,a_2,\dots,a_n\in L^{(m)}$. It is straightforward to verify that
there exists some $s\in \N$ such that $$L^{(s)}\subset
\span\{a_1a_2\cdots a_n|a_i\in L^{(m)}\}\subset U(L^{(m)}),$$ which
completes the proof.
\end{proof}

The simple modules in $\mathcal{O}_{\W}$ are studied in \cite{LZ3}.

\begin{lemma}\label{char-O(w)}\cite{LZ3} (a). Suppose that $A\in \mathcal{O}_{\mathfrak{a}}$ is
simple and nontrivial. Then there exists some $r\in \Z_+$ such that
$d_r A=0, \forall i>r$ and $d_r$ acts bijectively on $A$.
Consequently, $A$ is a simple $\mathfrak{a}_r$-module for some
$r\in\N$.

(b). Let $A\in \mathcal{O}_{\mathfrak{a}}$,
$W,W_1\in\mathcal{O}_{\W}$ be all nontrivial simple modules.
\begin{enumerate} \item The $\W$-module $\Ind_{\mathfrak{a}}^{\W}(A)$ is
simple   in $\mathcal{O}_{\W}$;
\item The $\mathfrak{a}$-module $\Soc_{\mathfrak{a}}(W)\in\mathcal{O}_{\mathfrak{a}}$ is  simple, and
an essential $\mathfrak{a}$-submodule of $W$, i.e. the intersection
of all nonzero $\mathfrak{a}$-submodules of $V$;
\item We have $W\cong\Ind_{\mathfrak{a}}^{\W}\Soc_{\mathfrak{a}}(W)$ and $
A= \Soc_{\mathfrak{a}}(\Ind_{\mathfrak{a}}^{\W}A)$;
\item We have $W\cong W_1$ if and only if $\Soc_{\mathfrak{a}}(W)\cong
\Soc_{\mathfrak{a}}(W_1)$.\end{enumerate} Consequently, $W$ is the
induced module from a simple $\mathfrak{a}_r$-module for some
$r\in\N$.
 \end{lemma}

Let us recall some results  for Whittaker modules over $\Vir$
studied in \cite{OW} and \cite{LGZ}.

For any nonnegative integer $m\in \Z_+$,  let $\psi_m: \V^{(m)}+\C
z_1\rightarrow \C$ be a Lie algebra homomorphism. Then we have the
one dimensional module $\C_{\psi_m}=\C w_{\psi_m}$ over $\V^{(m)}+\C
z_1$  with $x\cdot w_{\psi_m}=\psi_m(x)w_{\psi_m},\forall x\in
\V^{(m)}+\C z_1$. The induced $\V$-module
\begin{equation}\label{whittaker-1}W_{\psi_m}=\Ind_{\V^{(m)}+\C z_1}^{\V} \C_{\psi_m}\end{equation} is
called the universal Whittaker module with respect to $\psi_m$.

\begin{lemma}\cite{LGZ} For any $m\ge 1$, $W_{\psi_m}$ is simple if and only if $(\psi(d_{2m}),\psi(d_{2m-1}))\ne (0,0)$.\end{lemma}

Let us recall a result on   Whittaker modules over $\H$ from
\cite{C}. Suppose that $\theta:\C[t]+\C z_3 \rightarrow \C$ is a
linear map. Then $\C w_{\theta}$ becomes a one dimensional $\C[t]+\C
z_3$ module defined by $x w_{\theta}=\theta(x)w_{\theta}$ for all
$x\in \C[t]+\C z_3$. The induced $\H$-module
$W_{\theta}=\Ind_{\C[t]+\C z_3}^{\H} \C w_{\theta}$ is called a
Whittaker module with respect to $\theta$.

\begin{lemma} \label{H-whittaker}\cite{C} The $\H$-module $W_{\theta}$ is  simple  if and only if $\theta(z_3)\ne 0$.\end{lemma}

The simple modules in $\mathcal{O}_{\Vir }$ are studied in
\cite{MZ}. From Lemma \ref{eq-condition} and Theorem 2 in \cite{MZ},
we have

\begin{lemma}\cite{MZ}\label{char-O(v)}  Let $V\in\mathcal{O}_{\Vir }$ be simple.
\begin{enumerate} \item The $\mathfrak{a}$-module $\Soc_{\mathfrak{a}}(V)$ is simple, which is an essential $\mathfrak{a}$-submodule of
$V$, i.e. the intersection of all nonzero $\mathfrak{a}$-submodule
of $V$;
\item If $V$ is not a highest weight module, then
$V\cong\Ind_{\mathfrak{a}+\C z}^{\V}\Soc_{\mathfrak{a}}(V)$, where
the action of $z$ is a scalar;
\item Suppose $V, V_1\in \mathcal{O}_{\Vir }$. Then $V\cong
V_1$ if and only if $\Soc_{\mathfrak{a}}(V)\cong
\Soc_{\mathfrak{a}}(V_1)$ and $z$ acts on $V$, $V_1$ as the same
scalar.
\end{enumerate}
\end{lemma}

\

This paper is organized as follows. In Sect.2, we obtain a general result on
sufficient conditions for tensor product modules  to be simple over
an arbitrary Lie algebra. This result can be applied to many known
cases and several cases in this paper. In Sect.3, we   first
classify simple modules in $\mathcal{O}_{\H}$, and then construct
generalized oscillator representations of $\LL$ by extending the the
$\H$-module structure on simple modules in $\mathcal{O}_{\H}$. Many
simple modules in $\mathcal{O}_{\LL}$ with nonzero action of $z_3$
are proved to be decomposed into a tensor product of an oscillator
representation of $\LL$ and a simple $\Vir$-module. Then we apply
this theory to completely determine conditions for Whittaker modules over $\LL$ (in the more general setting as in \cite{BM}) to be
simple. In Sect.4, we use the ``shifting technique" to determine the
necessary and sufficient conditions for  the tensor products of
highest weight modules and modules of intermediate series over $\LL$
to be simple. In Sect.5,  we   first classify simple modules in
$\mathcal{O}_{\WW}$, and then establish the ``embedding trick"
 to make these
simple $W\in \mathcal{O}_{\WW}$ into simple $\LL$-modules
$W[\lambda]$ for any $\lambda\in\C^*$. By taking tensor product, we
obtain more simple $\LL$-modules.

\section{Simplicity of tensor product modules}

In this section, we will prove some general results for tensor
product modules, which will be frequently used later.

Let $V$ be a module over a Lie algebra $L$. For any $v\in V$, the
annihilator of $v$ is defined as $\ann_{L}(v)=\{g\in L| g v=0\}$.
For any $S\subset V$, define $$\ann_L(S)=\cap_{v\in S}\ann_{L}(v).$$

\begin{lemma}\label{density} Let $L$ be a Lie algebra over $\C$ with a countable basis, and   $V$ be a  simple $L$-module. For any $n\in \Z_+$ and any linearly independent subset
$\{v_1,v_2,\ldots v_n\}\subset V$, and any subset
$\{v_1',\ldots,v_n'\}\subset V$, there exists some $u\in U(L)$, such
that
 $$uv_i=v_i',\forall i=1,2,\ldots, n.$$
 \end{lemma}
 \begin{proof}
 Denote $R=U(L)/\ann_{U(L)}(V)$.  Then $R$ is an associative algebra with countable
 basis. It is well known that any endomorphism of a simple module over a countably generated associative $\C$-algebra
  is a scalar (Proposition 2.6.5 in [D]). Thus ${\rm Hom}_{R}(V,V)\cong
 \C$. Note that $V$ is a faithful and simple $R$-module. From
 the Jacobson Density Theorem (Page 197, [J]), we know that $R$ is isomorphic to a dense
 ring of endomorphisms of the $\C$-vector space $V$. The Lemma follows.
\end{proof}

Now we can give some useful sufficient conditions for a tensor product module to be simple.

\begin{theorem}\label{thm1} Let $L$ be a Lie algebra over $\C$ with a  countable basis, and  $V_1, V_2$ be $L$-modules.
Suppose that one of the following conditions holds:

\begin{enumerate}\item The module $V_1$ is simple and $\ann_L(v)+\ann_L(S)=L$ for all $v\in V_1$
and all finite subsets $S\subset V_2$;
\item
 For any finite subset
$S\subset V_2$, $V_1$ is a simple
$\ann_L(S)$-module.\end{enumerate} Then
\begin{enumerate}[$($a$)$] \item  Any
submodule of $V_1\otimes V_2$ is of the form $V_1\otimes V_2'$, for
a submodule $V_2'$  of $V_2$; \item If $V_1, V_2$ are simple, then
$V_1\otimes V_2$ is simple.\end{enumerate}
\end{theorem}

\begin{proof} (1) $\Rightarrow$ (2). Suppose that Condition 1 holds. For any nonzero $v\in V_1$ and a finite subset $S$ in $V_2$,
 we have
$$V_1=U(L)v=U(\ann_L(v)+\ann_L(S))v\hskip 3cm$$ $$=U(\ann_L(S))U(\ann_L(v))v=(U(\ann_L(S))v,$$
So $V_1$ is simple as $\ann_L(S)$-module. Thus Condition 2 holds.

\

Now we suppose that Condition 2 holds.

(a) Let $M$ be a nonzero submodule of $V_1\otimes V_2$. For any
nonzero $X\in M$, write $X=\sum_{i=1}^s v_{1,i}\otimes v_{2,i}\in
 M$ with minimal $s$. Then $\{v_{1,1},\ldots,v_{1,s}\}$ and $\{v_{2,1},\ldots, v_{2,s}\}$
 are linearly independent sets.

  Denote $L_2=\ann_{L}(\{v_{2,j}|j=1,2\ldots,s\})$. Then $V_1$ is a simple  $L_2$-module.
Now from Lemma \ref{density}, there exists some $u_0\in U(L_2)$ such
that
  $u_0v_{1,1}=v_{1,1}$, and $u_0v_{1,i}=0$ for $i=2,\ldots,s$.
  So $$u_0X=u_0(\sum_{i=1}^s v_{1,i}\otimes
v_{2,i})=\sum_{i=1}^s (u_0v_{1,i}\otimes v_{2,i})=v_{1,1}\otimes
v_{2,1}\in M.$$ For any $u\in U(L_2)$, we have $(uv_{1,1})\otimes
v_{2,1}=u(v_{1,1}\otimes v_{2,i})\in M.$ Thus $V_1\otimes
v_{2,1}=(U(L_2)v_{1,1})\otimes v_{2,1}\subset M$. Similarly we have
\begin{equation}V_1\otimes
v_{2,i}\subset M, \forall i=1,2,\ldots,s.\end{equation}

We have proved that
\begin{equation}M=V_1\otimes V_2',\end{equation}
where $V_2'=\{v_2\in V_2| V_1\otimes v_2\subseteq M\}$. For any
$v_1\in V_1$, $v_2\in V_2'$, and for all $g\in L$, we have
$v_1\otimes g v_2=g(v_1\otimes v_2)-gv_1\otimes v_2\in M$. So
 $V_2'$ is an $L$-submodule of $V_2$.

Part (b) follows from (a).\end{proof}

\begin{example} Theorems 3.3 and 3.4 in \cite{GZ} can follow from our Theorem 7. For other applications, see Theorems  11, 13, 33.
\end{example}

\begin{lemma}\label{induced}Let $W$ be a module over a Lie algebra $L$, $\mathfrak{g}$ be a subalgebra of $L$, and $B$ be a $\mathfrak{g}$-module.
Then the $L$-module homomorphism $\tau:
\Ind_{\mathfrak{g}}^L(W\otimes B)\rightarrow W\otimes
\Ind_{\mathfrak{g}}^L(B)$ induced from the inclusion map $W\otimes
B\rightarrow W\otimes \Ind_{\mathfrak{g}}^L(B)$ is an $L$-module
isomorphism.\end{lemma}
\begin{proof} Take a subspace $V\subset L$ such that $L={\mathfrak{g}}\oplus V$. Let $\{l_i|i\in I\}$, $\{b_j|j\in J_1\}$, $\{w_j|j\in J_2\}$ be bases of $V$, $B$ and $W$,
respectively. Denote $S=\{l_{i_1}^{k_1}l_{i_2}^{k_2}\cdots
l_{i_m}^{k_m}|i_1\prec i_2\prec\cdots \prec k_m\}$, where $\prec$ is
a total order on $I$, and $U_n=\span
\{l_{i_1}^{k_1}l_{i_2}^{k_2}\cdots l_{i_m}^{k_m}\in
S\,\,|\,\,k_1+\ldots+k_m=n\}$. Then $$T_1=\{x (w_s\otimes b_t)
\,\,|\,\, x\in S, s\in J_1,t\in J_2\},$$
$$T_2=\{w_{s}\otimes (x b_t)
\,\,|\,\,x\in S, s\in J_1,t\in J_2\}$$ are bases of
$\Ind_{\mathfrak{g}}^L(W\otimes B)$ and $W\otimes
Ind_{\mathfrak{g}}^L(B)$, respectively.

 For
any $x=l_{i_1}^{k_1}l_{i_2}^{k_2}\cdots l_{i_m}^{k_m}\in S$, by
induction on $n=\sum_{i=1}^m k_i$ we can prove that
\begin{equation}\tau(x (w_s\otimes b_t))\in w_s\otimes
(x b_t)+W\otimes \sum_{i<n} (U_i B),\end{equation} from which we may
easy deduce that $\tau$ is bijective.
\end{proof}

\section{Generalized oscillator representations of  $\LL$}

In this section we will first classify simple modules in
$\mathcal{O}_{\H}$, and then construct generalized oscillator
representations of $\LL$ by extending the the $\H$-module structure
on simple modules in $\mathcal{O}_{\H}$. Since the representation
space is in general not the Fock space, we call the resulting
representation as generalized oscillator representations of $\LL$.
When the representation space is indeed the Fock space, we actually
obtain the usual  oscillator representations of $\LL$. This
generalizes the construction in Sect.2.3 of \cite{KR}.

\subsection{Simple modules in $\mathcal{O}_{\H}$}

For any $m\in\N$, we define $$\mathfrak{T}_m=\span\{I_i,
z_3|i=-m,-m+1,\ldots,m\},$$
       $$\H_{m}=\span\{I_i,z_3|i\ge -m, i\in \Z\}.$$

\begin{proposition}Let $B$, $B'$ be simple modules over $\mathfrak{T}_m$ for some $m\in \N$  with  nonzero action of $z_3$.

(a). The $\H$-module $\Ind_{{\H}_{m}}^\H B$ is simple, where $B$ is
regarded as $\H_{m}$-module by $I_i B=0, \forall i>m$. Moreover, all
nontrivial simple modules in $\mathcal{O}_\H$ can be obtained in
this way.

(b). As $\H$-modules, $\Ind_{\H_{m}}^\H B\cong \Ind_{\H_{m}}^\H B'$
if and only if $B\cong B'$ as $\mathfrak{T}_m$-modules.
\end{proposition}

\begin{proof} Let $K_{m}=\span\{I_{-m-j}|j\in \N\}.$

(a). Using PBW Theorem and nonzero action of $z_3$, we can easily
prove that the $\H$-module $\Ind_{\mathcal{\H}_{m}}^\H B =U(K_m)B$
is simple.

Now suppose that $V\in \mathcal{O}_\H$ is simple and nontrivial.
Take a nonzero $v\in V$ such that  $m\in\N$ is minimal with
$I_{m+j}v=0$ for all $j\in\Z_+$.

{\bf Claim 1.} $U(\H_m)v$ is a simple $\H_m$-module.

Since $V=U(K_m)U(\H_m)v=U(K_m)U(\mathfrak{T}_m)v$ is nontrivial, we
see that $U(\mathfrak{T}_m)v\not\subseteq \C v$, hence the action of $z_3$ is
nonzero. We can deduce that $V$ is a free $U(K_m)$-module on
$U(\H_m)v$. Since $V$ is a simple $\H$-module, we deduce that
$B=U(\H_m)v$ is a simple $\H_m$-module.

Thus $V=\Ind_{{\H}_{m}}^\H B$.

(b). Noting that $B$ and $ B'$ are the socles of $\H_m$-modules
$\Ind_{\H_{m}}^\H B$ and $\Ind_{\H_{m}}^\H B'$ respectively, we see
that, if $\Ind_{\H_{m}}^\H B\cong \Ind_{\H_{m}}^\H B'$  as
$\H_m$-modules, hence  $B\cong B'$ as $\mathfrak{T}_m$-modules. The
converse is trivial.
\end{proof}

From the above theorem, to classify all simple modules over $\H$ is
equivalent to classifying all simple modules over the
finite-dimensional  Heisenberg algebras $\mathfrak{T}_m$ for all
$m\in \N$, which is equivalent to classifying all simple modules
over the rank $m$ Weyl algebras $A_m$. Such a classification is only
known for $m=1$, see \cite{Bl}.

\begin{example} Take $m=1$, make $\C v$ into a module over $\mathfrak{b}=\C I_0+\C I_{-1}+\C z_3$ by  $I_0v=\dot I_0v, I_{-1}v=av, z_3v=\dot z_3v$ for $a,\dot I_0, \dot z_3\in\C$ with $\dot z_3\ne0$. We have the simple $\mathfrak{T}_1$-module $B=\Ind^{\mathfrak{T}_1}_\mathfrak{b}\C v$. Then we obtain the simple weight $\H$-module $\Ind_{\mathcal{\H}_{m}}^\H B$ which has all nonzero weight spaces infinite-dimensional. In this manner similar to Sect.3.1 in \cite{BBFK}, one can construct a lot of simple weight modules in $\mathcal{O}_{\H}$.
\end{example}

\begin{example} Let $B=\C[t, t^{-1}]$ be a simple $\mathfrak{T}_1$-module defined by
$I_{-1}t^n=t^{n+1}, I_1t^n=t^{n-1}(\alpha+n), I_0t^n=at^n, z_3t^n=t^n$ for $\alpha\in\C\setminus\Z$ and $a\in\C$.
Then we obtain the simple  $\H$-module $\Ind_{\mathcal{\H}_{m}}^\H B$. In this manner using simple modules over the Weyl algebra as in \cite{Bl}, one can construct a lot of simple modules in $\mathcal{O}_{\H}$.
\end{example}

\subsection{Generalized oscillator representations of  $\LL$}

In this subsection, we prove that for any $\dot{z_2}\in \C$, a
simple representation $H\in \mathcal{O}_{\H}$ with nonzero action of
$z_3$  can always be extended to a representation of $\LL$ with
$z_2$ acting as scalar $\dot{z_2}$. Here we will use oscillator-like
representations on $H$.

Let $\sigma$ be an endomorphism of some Lie algebra $L$, and $V$ be
any module of $L$. We can make $V$ into another $L$-module, by
defining the new action of $L$ on $V$ as

\begin{equation}x \circ v = \sigma(x) v, \forall x\in L,v\in V.\end{equation}

We will call the new module as {\it the twisted module of $V$ by}
$\sigma$, and denote it by $V^{\sigma}$. Two $L$-modules $V$ and $W$
are said to be equivalent if there exists some automorphism $\sigma$
of $L$ such that $V=W^{\sigma}$.

For any \begin{equation}\a=\sum_{i\in \Z}a_it^i\in \C[t,t^{-1}],
b\in \C,\end{equation} we have the $\sigma=\sigma_{\a,b}\in
\Aut(\LL)$ defined as
$$\sigma(d_n)=d_n+t^n(\a+nb)-(n+1)a_{-n}z_2\hskip 3cm$$
\begin{equation} -(\frac{\sum_ia_ia_{-n-i}}{2}
+a_{-n}nb)z_3+\delta_{n,0}b(z_2+\frac{b}{2}z_3)\end{equation}
\begin{equation}\sigma(t^n)=t^n+\delta_{n,0}bz_3-a_{-n}z_3,\sigma(z_1)=z_1-24bz_2-12b^2z_3,\end{equation}
\begin{equation}\sigma(z_2)=z_2+bz_3, \sigma(z_3)=z_3. \end{equation}
For more details, see \cite{LGZ}.

\begin{theorem}Let $\dot{z_2}, \dot{z_3}\in\C$ with $\dot{z_3}\ne0$. Let $H\in \mathcal{O}_\H$ be a simple module with $z_3$ acting
as a   scalar $\dot{z_3}$. Then the   $\H$-module structure  on $H$
can be extended to an $\LL$-module by

\begin{equation}z_1=1-\frac{12\dot{z_2}^2}{\dot{z_3}}, z_2=\dot{z_2}, z_3=\dot{z_3}, \end{equation}
\begin{equation}\label{dk}d_k=-\frac{1}{2\dot{z_3}}\sum_{i\in \Z}:I_{-i} I_{i+k}:+\frac{(k+1)\dot{z_2}}{\dot{z_3}}I_k,\forall k\in \Z.\end{equation}
where,  for all $i,j\in \Z$, the normal order is defined as $$:I_{i}
I_{j}:=\left\{
                                       \begin{array}{cc}
                                        I_i I_j,&\mbox{if}\,\, i<j, \\
                                       I_j I_i, & \,\,\mbox{otherwise.}\\
                                     \end{array}
                                    \right.$$
The resulting $\LL$-module structure on $H$ will be denoted by
$H(\dot{z_2})$.\end{theorem}

\begin{proof}  Let  $b=\frac{\dot{z_2}}{\dot{z_3}}$. Note that the righthand side in (\ref{dk}) is well defined as an operator on $H$, i.e.,
$-\frac{1}{2\dot{z_3}}\sum_{i\in \Z}:I_{-i} I_{i+k}:v+b(k+1)I_kv$
makes sense for any $v\in H$ since there are only finitely many
nonzero terms. It is straightforward to check that
$H(\dot{z_2})=H(0)^{\sigma_{b,b}}$.

So we only need to verify that $H(0)$ is a $\LL$ module. The proof
is similar to the arguments in Section 2.3 of \cite{KR}. The
verifications are tedious but straightforward. We omit the details.
\end{proof}

We see that (3.7) is the  oscillator-like actions on $H$. So we call
the $\LL$-module $H(\dot z_2)$ as a generalized oscillator
representation of $\LL$.

Note that for any $v\in H\in \mathcal{O}_\H$ with $z_3\ne 0$ and
$\dot{z_2}\in \C$, if $I_i v=0, \forall i\ge n$ for some $n\in \N$,
then in $H(\dot{z_2})$, we have
\begin{equation}d_i v=0,\forall i\ge 2n.\end{equation} We see that
as an $\V$-module, $H(\dot{z_2})\in \mathcal{O}_\V$ which will not
give new simple $\V$-modules.

If $H$ is a weight $\H$-module, then $H(\lambda)$ is a
weight module. When $H$ is a highest weight module over $\H$, the
$\V$-module $H(\dot{z_2})$ was constructed in Sect.3.4 of \cite{KR}.
For any $\lambda\in \C$, the $\LL$-module $H(\lambda)$ is simple if
and only if $H$ is simple as $\H$-module.

Note that the action of $z_1$ on  $H(\dot z_2)$ is determined by $z_2$ and $z_3$. In the next subsection we will give a method so that the action $z_1$ can be arbitrary.

\begin{example} Let $H$ be  the simple  $\H$-module $\Ind_{\mathcal{\H}_{1}}^\H B$
constructed in Examples 2 and 3. Using Theorem 10 we can obtain
simple $\LL$-modules $H(\dot z_2)$. From the simple weight modules
$H$ defined in Example 2, we obtain simple weight modules over $\LL$
whose nonzero weight spaces are infinite-dimensional.
\end{example}

\begin{example}Let $\dot z_2, \dot z_3, \dot I_0, a\in \C$.
Define $\theta(z_3)=\dot z_3\ne0, \theta(I_0)=\dot I_0,
\theta(I_1)=a, \theta(I_i)=0$ for $i>1$. Then we have the simple
$\H$-module $H=W_\theta$ in
 Lemma 4. From Theorem 10 we have the
simple $\LL$-module $H(\dot z_2)$ which is a Whittaker module
defined in \cite{LWZ}. Unlike Corollary 4.5 in  \cite{LWZ}, the
$\LL$-module $H(\dot z_2)$ is not free over $U(\LL^-)$.
\end{example}

\subsection{Simple modules from tensor product}

Let $V$ be a $\Vir$-module. Then we can regard $V$ as $\LL$-module
by defining $(\C z_2+\H)V=0$. The resulting $\LL$-module  will be
denoted by $V^{\LL}$.

\begin{theorem}\label{thm-10}Suppose that $\lambda, \mu \in \C$, $V, W$ are  $\Vir$-modules, and $H,K\in \mathcal{O}_\H$ are simple with   nonzero action of $z_3$.
 \begin{enumerate}\item  Any  $\LL$-submodule of  $V^{\LL}\otimes H(\lambda)$ are of the
form $(V')^{\LL}\otimes  H(\lambda)$ for  a $\Vir$-submodule $V'$ of
$V$. In particular, $V^{\LL}\otimes H(\lambda)$ is simple as
$\LL$-module if and only if $V$ is a simple $\Vir$-module.
\item $V^{\LL}\otimes H(\lambda)\cong W^{\LL}\otimes K(\mu)$ if and
only if $\lambda=\mu, V\cong W$, and $H\cong K$. \end{enumerate}
 \end{theorem}
 \begin{proof}(1) Note that $\H\subset \ann_{\LL}({V})$, and ${H(\lambda)}$ is  a simple $\H$-module. Thus the theorem follows from
 Theorem \ref{thm1}.

 (2) The sufficiency is trivial. Now suppose that $\psi: V^{\LL}\otimes H(\lambda)\rightarrow W^{\LL}\otimes K(\mu)$ is a $\LL$-module isomorphism. By comparing the action of $z_2$, we have $\lambda=\mu$.
 Now for any nonzero element $x\in V^{\LL}\otimes H(\lambda)$, it is clear that $U(\H)x\cong H$ as $\H$-module. Similarly we have $U(\H)y\cong K$ for any $0\ne y\in W^{\LL}\otimes K(\mu)$. Thus $H\cong K$ as $\H$-module.
 Without lose of generality, we may assume that $H=K$. Then $\psi$ becomes an $\LL$-module isomorphism from $V^{\LL}\otimes H(\lambda)$ to $W^{\LL}\otimes H(\lambda)$.
  Now for any nonzero $v\in V, h\in H$, write $\psi(v\otimes h)=\sum_{i=1}^s w_i\otimes k_i$ with minimal $k$. Then by Lemma \ref{density},
 there exists some $u\in U(\H)$ such that $uk_i=\delta_{1,i}k_i,i=1,2,\ldots,s$. Therefore $\phi(v\otimes uh)=\phi(u(v\otimes h))=w_1\otimes k_1$.
 Now write $\phi(v\otimes h)=w_1\otimes \tau_{v}(h)$. Then it is easy to check that $\tau_{v}:H\rightarrow H$ is a $U(\H)$-module
 automorphism for any $v\in V$. Recall that any endomorphism of a simple module over a countably generated associative $\C$-algebra
  is a scalar. Hence there exists a linear map $\nu: V\rightarrow W$ such that  $\phi(v\otimes h)=\nu(v)\otimes h$ for all $v\in V, h\in H$. Now
 from $d_n\phi(v\otimes h)=\phi(d_n (v\otimes h))$, we have $d_n\nu(v)\otimes h=\nu(d_n v)\otimes h$, $\forall v\in V, h\in H.$
 Hence $d_n\nu(v)=\nu(d_n(v))$. So $\nu:V\rightarrow W$ is a $\Vir$-module isomorphism.
   \end{proof}

\begin{theorem} Let $V\in \mathcal{O}_{\LL}$ be simple with nonzero action of $z_3$. If $V$ contains a simple $\H$-submodule $H$, then $V\cong H(\dot{z_2})\otimes U^{\LL}$ as $\LL$-modules for some $\dot{z_2}\in\C$ and some simple module $U\in \mathcal{O}_\mathcal{V}$.
\end{theorem}
\begin{proof} Suppose that $I_0, z_1,z_2,z_3$ act on $V$ as scalars $\dot{I_0}, \dot{z_1},\dot{z_2},\dot{z_3}$, where we have assumed that
$\dot{z_3}\ne 0$. Let $0\ne v\in H$ and $n_1, n_2\in \N$ with $d_i
v=0,$ for all $i\ge n_1$ and $I_j v=0,$ for all $j\ge n_2$. Denote
$n=\max\{n_1, 2n_2\}$. It is not hard to show that $H$ is a simple
module over $L=\span\{d_i,I_j,z_1,z_2,z_3|i\ge n,j\in \Z\}$. Define
the one dimensional $L$-module $\C v_0$ by $d_i,I_j,z_2,z_3$ acting
as zero for all $i\ge n, j\in \Z$, and $z_1$ acting as
$\dot{z_1}+\frac{12\dot{z_2}^2}{\dot{z_3}}-1$. It is clear that
$H\cong H(\dot{z_2})\otimes \C v_0$ as $L$-module. Now from Lemma
\ref{induced}, we have $$\Ind_L^{\LL}(H)\cong \Ind_{L}
(H(\dot{z_2})\otimes \C v_0)\cong H(\dot{z_2})\otimes \Ind_{L}^\LL
\C v_0.$$ Note that $\Ind_{L}^\LL \C v_0\in
\mathcal{O}_\mathcal{V}$, and that $V$ is a simple quotient module
over $\Ind_L^{\LL}(H)$. Now the theorem follows from Theorem
\ref{thm-10}.\end{proof}

{\bf Open Problem:} It will be interesting to classify all simple modules in $\mathcal{O}_{\LL}$.

\subsection{Whittaker modules over $\LL$}

Now we will focus on the so called Whittaker modules over $\LL$.

For any $m\in \Z_+$, recall $\LL^{(0,m)}=\span\{d_{m+i},t^i|i\in
\Z_+\}$. Let $\phi_m: \LL^{(0,m)}+\sum_{i=1}^3 \C z_i\rightarrow \C$
be a Lie algebra homomorphism. Then we have the one dimensional
module $\C_{\phi_m}=\C w_{\phi_m}$ over $\LL^{(0,m)}+\sum_{i=1}^3 \C
z_i$ with $x\cdot w_{\phi_m}=\phi_m(x)w_{\phi_m},\forall x\in
\LL^{(0,m)}+\sum_{i=1}^3 \C z_i$. The induced $\LL$-module
\begin{equation}\label{whittaker}\widetilde{W}_{\phi_m}=\Ind_{\LL^{(0,m)}+\sum_{i=1}^3 \C
z_i}^{\LL} \C_{\phi_m}\end{equation} will be called the universal
Whittaker module with respect to $\phi_m$. And any nonzero quotient
of $\widetilde{W}_{\phi_m}$ will be called a Whittaker module with
respect to $\phi_m$.

 Note that $\phi_m([\LL^{(0,m)},\LL^{(0,m)}])=0$.  Then we have $$\phi_m(d_{2m+j})=\phi_m(I_{m+j})=0, \forall j\in
 \N.$$  And if $m=0$, $\widetilde{W}_{\phi_m}$ will be a highest weight
module.

For the above $\phi_m$, we define a new Lie algebra homomorphism
$\phi'_m: \V^{(m)}+\C z_1  \rightarrow \C$ as follows

$$\phi'_m(z_1)=\phi_m(z_1)-1+12\frac{\phi_m(z_2)^2}{\phi_m(z_3)},$$
$$\phi'_m(d_{k})=0,\forall k\ge 2m+1,$$
$$\phi'_m(d_k)=\phi_m(d_k)+\frac{(\sum_{i=0}^k
\phi_m(I_i)\phi_m(I_{k-i}))-2(m+1)\phi_m(I_m)\phi_m(z_2)}{2\phi_m(z_3)},$$
for all $k=m,m+1, \ldots,2m.$ Then we have  the universal whittaker
$\Vir$-module $W_{\phi'_m}$:
\begin{equation}\label{whittaker}{W}_{\phi'_m}=\Ind_{\V^{(m)}+
\C z_1}^{\LL} \C w_{\phi'_m}\end{equation} where $x\cdot
w_{\phi'_m}=\phi'_m(x)w_{\phi'_m},\forall x\in \V^{(m)}+ \C z_1$.

\begin{theorem}\label{thm-whittaker-1} Suppose that $m\in \Z_+$, and $\phi_m$ and $\phi'_m$ are given above with $\phi_m(z_3)\ne
0$. Let $H=U(\H)w_{\phi_m}$ in $W_{\phi_m}$.

\begin{enumerate}
\item We have $\widetilde{W}_{\phi_m}\cong H(\phi_m(z_2))\otimes W_{\phi'_m}^{\LL}$.
Consequently, each simple whittaker module with respect to $\phi_m$
is isomorphic to   $H(\phi_m(z_2))\otimes T^{\LL}$ for a simple
quotient $T$ of $W_{\phi'_m}$.
\item  The  $\LL$-module $\widetilde{W}_{\phi_m}$ is  simple if and only if
$W_{\phi'_m}$ is a simple $\Vir$-module. Consequently, if $m\in \N$, then $\widetilde{W}_{\phi_m}$ is simple if and only if
$(\phi'_m(d_{2m-1}),\phi'_m(d_{2m}))\ne(0,0)$, i.e.,
$$2\phi_m(d_{2m})\phi_m(z_3)+\phi_m(I_{m})^2-2(m+1)\phi_m(I_{m})\phi_m(z_2)\ne0,\,\,{\text{or}}
$$
$$\phi_m(d_{2m-1})\phi_m(z_3)+\phi_m(I_{m})\phi_m(I_{m-1})-(m+1)\phi_m(I_{m})\phi_m(z_2)\ne0.
$$
\item Let $T_1, T_2$ be simple quotients of $W_{\phi'_m}$. Then
$H(\phi_m(z_2))\otimes T_1^\LL\cong H(\phi_m(z_2))\otimes T_2^\LL$
if and only if $T_1\cong T_2$.\end{enumerate}
\end{theorem}
 \begin{proof} From Lemma \ref{H-whittaker}, we know that $H$ is a simple  $\H$-module.

 (1)   Define $L=\span\{d_{m+i},I_j,z_1,z_2,z_3|i\in \Z_+,j\in \Z\}$.
 From simple computations we see that  $$H\cong \Ind_{\LL^{(0,m)}+\sum_{i=1}^3 \C z_i}^L \C w_{\phi_m}\cong H(\dot{z_2})
 \otimes \C w_{\phi'_m}$$ as $L$-modules, where the action of $L$ on $\C w_{\phi'_m}$ is given by
 $(\H+\C z_2+\C z_3) w_{\phi'_m}=0, xw_{\phi'_m}=\phi'_m(x)w_{\phi'_m}$ for all $x\in \V^{(m)}+\C z_1$. Therefore from Lemma
 \ref{induced}, we have $$\widetilde{W}_{\phi_m}\cong \Ind_L^{\LL}(\Ind_{\LL^{(0,m)}+\sum_{i=1}^3 \C z_i}^L \C
 w_{\phi_m})\cong \Ind_L^{\LL}( H(\dot{z_2})\otimes \C
 w_{\phi'_m})$$ $$\cong H(\dot{z_2})\otimes \Ind_L^{\LL} \C
 w_{\phi'_m})\cong H(\dot{z_2})\otimes W_{\phi'_m}^{\LL}.$$

Parts (2) and (3) follows from Theorem \ref{thm-10} and some easy
computations.
\end{proof}

Next we consider $\widetilde{W}_{\phi_m}$ for $\phi(z_3)=0$.

\begin{theorem}\label{thm-whittaker-2} Suppose that $m\ge 1$ and  the Lie algebra homomorphism
$\phi_m: \LL^{(0,m)}+\sum_{i=1}^3 \C z_i\rightarrow \C$ is  with
$\phi_m(z_3)=0$. Then the universal Whittaker module
$\widetilde{W}_{\phi_m}$ is simple if and only if $\phi_m(I_m)\ne
0$.
\end{theorem}

\begin{proof} {\bf Case 1.} $\phi_m(I_m)\ne 0$.

 Since $\phi_m(d_{2m+j})=\phi_m(I_{m+j})=0$ for all $j\in \N, $
we may choose some $\a=\sum_{i=-m}^{-1}a_i t^i\in \C[t,t^{-1}]$
 such that $$0=\phi_m(d_n)+\phi_m(t^n\a), \forall n=m+1,m+2,\ldots,2m.$$
 Then $\widetilde{W}_{\phi_m}^{\sigma_{\a,0}}$ becomes a new Whittaker module with the new action
 $d_k\circ w_{\phi_m}=0$ for all $k>m$. Since $\phi(z_3)=0$,  the action of $\H$ on $\widetilde{W}_{\phi_m}$ is unchanged.
 Without lose of generality, we may assume that
 \begin{equation}\phi_m(d_k)=0,\forall k>m.\end{equation}
 Denote $B=\Ind_{\LL^{(0,m)}}^{\LL^{(0)}} \C w_{\phi_m}$. Then
 $\{d_{m-1}^{i_{m-1}}\cdots
 d_0^{i_0}w_{\phi_m}|(i_{m-1},\ldots, i_0)\in \Z_+^m\}$ is a basis
 of $B$. It is straightforward to check that $I_k$ acts injectively on $B$, and $B$ is simple as $\LL^{(0,0)}$-module.
 Now from Theorem 1 in \cite{CG} (for $d=0$ there), we see that $W_{\phi_m}\cong \Ind_{\LL^{(0,0)}+\sum_{i=1}^3 \C z_i}^{\LL} B$ is simple.

 {\bf Case 2.} $\phi_m(I_m)=0$.

Let
$$  w=(\sum_{i=0}^{m-1}a_i
d_i+\sum_{i=1}^{m+1}b_iI_{-i}+cI_{-1}I_{-1})w_{\phi_m}\in
\widetilde{W}_{\phi_m},$$
where $a_0,\ldots,a_{m-1},b_1,\ldots,b_{m+1},c\in\C$ will be determined.
We will make $w$ into a Whittaker vector. From $d_{m+i}w=\phi_m(d_{m+i})w$ and $I_{1+i}w=\phi_m(I_{1+i})w$ for $i=0, 1, ..., m$, we obtain the  following   homogenous linear
system with $2m+2$ variables
$a_0,\ldots,a_{m-1},b_1,\ldots,b_{m+1},c$ and $2m+1$ equations:
$$\aligned
                &\sum_{i=0}^{m-1}a_i\phi_m([I_k,d_i])=0,  k=1,\ldots,m-1; \\
                 & (-m-1)b_{m+1}+2c\phi_m([d_m,I_{-1}])=0,\\
                 &\sum_{i=0}^{m-1}a_i\phi_m([d_m,d_i])+\sum_{i=1}^m b_i\phi_m([d_m,I_{-i}])=0, \\
                & \sum_{i=0}^{m-1}a_i\phi_m([d_{m+l},d_i])+\sum_{i=1}^{m+1}b_i\phi_m([d_{m+l},I_{-i}])=0, l=1,2\ldots,m.
            \endaligned
$$

It is clear that the homogenous linear system has a nonzero solution since the number of equations is less than the number of variables.

Thus $w$ generates a nonzero proper  submodule of
$\widetilde{W}_{\phi_m}$. So we have proved that
$\widetilde{W}_{\phi_m}$ is not simple in this case.
 \end{proof}

Now we summarize the established results of this subsection into the following main
theorem.

\begin{theorem} Let  $m\in \N$ and
$\phi_m: \LL^{(0,m)}+\sum_{i=1}^3 \C z_i\rightarrow \C$ be a Lie algebra homomorphism.

 \begin{enumerate}
\item Suppose $\phi_m(z_3)\ne
0$. Then  the  Whittaker module $\widetilde{W}_{\phi_m}$ is  simple if and only if
$$2\phi_m(d_{2m})\phi_m(z_3)+\phi_m(I_{m})^2-2(m+1)\phi_m(I_{m})\phi_m(z_2)\ne0,\,\,{\text{or}}
$$
$$\phi_m(d_{2m-1})\phi_m(z_3)+\phi_m(I_{m})\phi_m(I_{m-1})-(m+1)\phi_m(I_{m})\phi_m(z_2)\ne0.
$$
\item  Suppose $\phi_m(z_3)=0$.  Then  the   Whittaker module $\widetilde{W}_{\phi_m}$ is  simple if and only if  $\phi_m(I_m)\ne
0$.
 \end{enumerate}
\end{theorem}

\section{Tensor products of  highest weight modules and
modules of intermediate series}

In this section, we will determine the necessary and sufficient
conditions for the tensor products of  highest weight modules and
modules of intermediate series over $\LL$ to be simple.

Let $(\dot{I_0},\dot{d_0},\dot{z_1},\dot{z_2},\dot{z_3})\in \C^5$,
$\mathcal{I}$ be the left ideal in $U(\LL)$ generated by $\{d_n,I_n,
d_0-\dot{d_0},I_0-\dot{I_0},
z_1-\dot{z_1},z_2-\dot{z_2},z_2-\dot{z_3}| n\in \N\}$. Then
$M(\dot{I_0},\dot{d_0},\dot{z_1},\dot{z_2},\dot{z_3})=U(H)/\mathcal{I}$
is a Verma module which is a free $U(\LL^-)$ module generated by the
highest weight vector $w=1+\mathcal{I}$. It has a unique maximal
submodule $J(\dot{I_0},\dot{d_0},\dot{z_1},\dot{z_2},\dot{z_3})$,
and the quotient module
$V(\dot{I_0},\dot{d_0},\dot{z_1},\dot{z_2},\dot{z_3})
=M(\dot{I_0},\dot{d_0},\dot{z_1},\dot{z_2},\dot{z_3})/J(\dot{I_0},\dot{d_0},\dot{z_1},\dot{z_2},\dot{z_3})$
  is simple. We will denote the image of $w$ by $\bar w$. These modules are studied in \cite{ACKP, B}.

Now we recall the modules of intermediate series from \cite{LZ1}.
For any $a,b,F\in \C$, $A(a,b;F)=\C[x,x^{-1}]$ is a $\LL$-module
with the action

\begin{equation}z_1=z_2=z_3=0,\end{equation}
\begin{equation}d_n x^m=(\a+m+nb)x^{m+1},\end{equation}
\begin{equation}I_n x^m=Fx^{m+n},\forall m,n\in \Z.\end{equation}

It is well known that $A(a,b;F)$ is not simple if and only if $a\in
\Z, b\in \{0,1\}$ and $F=0$. Denote by $A'(a,b;F)$ the unique
nontrivial simple sub-quotient of $A(a,b;F)$. Recall that
$A'(a,b,F)\cong A'(0,0,0)$ if $A(a,b;F)$ is not simple.

It has been shown in \cite{LZ1} that an simple $\LL$-module with
finite-dimensional weight spaces is either a highest (or lowest)
weight module, or isomorphic to some $A'(a,b;F)$.

\begin{lemma}\label{0-0-0} (1)  Suppose $\dot{z_3}\ne0$. Then
$V(\dot{I_0},\dot{d_0},\dot{z_1},\dot{z_2},\dot{z_3})\otimes
A'(a,b;0)$ is simple if and only if
$V(\dot{d_0}+\frac{1}{2\dot{z_3}}\dot{I_0}^2-\frac{\dot{z_2}}{\dot{z_3}}\dot{I_0},\dot{z_1}-1
+\frac{12\dot{z_2}^2}{\dot{z_3}})\otimes A'(a,b)$ is a simple
$\Vir$-module (which is determined in \cite{CGZ}).

(2) If $F\ne 0$, then $A(a,b;F)\otimes V(\dot{d_0},\dot{z_1})^\LL$
is simple.

\end{lemma}

\begin{proof} (1). Denote by $H$ the highest weight $\H$-module with $I_0=\dot{I_0}$ and $z_3=\dot{z_3}$.
From Theorem \ref{thm-whittaker-1} or \cite{ACKP} we know that
$$V(\dot{I_0},\dot{d_0},\dot{z_1},\dot{z_2},\dot{z_3})\cong
H(\dot{z_2})\otimes
V(\dot{d_0}+\frac{1}{2\dot{z_3}}\dot{I_0}^2-\frac{\dot{z_2}}{\dot{z_3}}\dot{I_0},\dot{z_1}-1
+\frac{12\dot{z_2}^2}{\dot{z_3}})^\LL,$$
$$\hskip -5cmV(\dot{I_0},\dot{d_0},\dot{z_1},\dot{z_2},\dot{z_3})\otimes
A'(a,b;0) $$ $$\hskip 2cm \cong H(\dot{z_2})\otimes
(V(\dot{d_0}+\frac{1}{2\dot{z_3}}\dot{I_0}^2-\frac{\dot{z_2}}{\dot{z_3}}
\dot{I_0},\dot{z_1}-1+\frac{12\dot{z_2}^2}{\dot{z_3}})\otimes
A'(a,b))^{\LL}.$$ Therefore the result follows from Theorem
\ref{thm-10}.

(2) Note that for any finite subset $S\subset
V(\dot{d_0},\dot{z_1})$, there exists some $r$ such that
$\Vir^{(r)}+\H+\C z_2\subset \ann_{\LL}(S)$. However $A(a,b;F)$ is
simple as $\Vir^{(r)}+\H+\C z_2$ module. Thus the result follows
from Theorem \ref{thm1}.
\end{proof}

\begin{lemma}\label{Verma-nonsimple} The module $M(\dot{I_0},\dot{d_0},\dot{z_1},\dot{z_2},\dot{z_3})\otimes A'(a,b;F)$ is not simple.\end{lemma}
\begin{proof}Let $w$ be the highest weight vector of $M(\dot{I_0},\dot{d_0},\dot{z_1},\dot{z_2},\dot{z_3})$.
Suppose that $x^{k},x^{k+1}\ne 0$ in $A'(a,b;F)$. The lemma is clear
from the following claim.

{\bf Claim.} We have $w\otimes x^{k}\notin U(\LL)(w\otimes
x^{k+1})$.

Note that  $$U(\LL)(w\otimes x^{k+1})=U(\LL^-)U(\LL^+)(w\otimes
x^{k+1})\subset \sum_{i\in \Z_+}U(\LL^-)(w\otimes x^{k+1+i}).$$
Using the PBW Basis of $U(\LL^-)$, it is easy to verify that
$w\otimes x^k$ can not be written as $\sum_{i=1}^r u_{-i}(w\otimes
x^{k+i})$, where $u_{-i}\in U(\LL^-)_{-i}$ with $u_{-r}\ne0$, since
the term $u_{-r}w\otimes x^{k+r}$ in the expression of $\sum_{i=1}^r
u_{-i}(w\otimes x^{k+i})$ cannot be canceled.
\end{proof}

\begin{corollary}\label{simplicity-00}If $\dot{I_0}\ne 0$, then $V(\dot{I_0},\dot{d_0},\dot{z_1},0,0)\otimes A'(a,b;F)$ is not simple.\end{corollary}
\begin{proof}From Theorem 1 in \cite{CG}, we have
$$V(\dot{I_0},\dot{d_0},\dot{z_1},0,0)=M(\dot{I_0},\dot{d_0},\dot{z_1},0,0)$$
 if $\dot{I_0}\ne 0$. The result follows from Lemma \ref{Verma-nonsimple}.\end{proof}

\begin{lemma}\label{lemma-19}Suppose that $F\ne 0$. Then any nonzero submodule $M$ of $V(\dot{I_0},\dot{d_0},\dot{z_1},
\dot{z_2}, \dot{z_3})\otimes A(a,b;F)$ contains $\bar{w}\otimes
x^{k}$ for some $k$.\end{lemma}

\begin{proof} Take a nonzero $ \b=\sum_{i=0}^s
v_{-i}\otimes x^{k+i}\in M$, where $v_{-i}\in U(\LL^-)_{-i}\bar{w}$.
Replacing $\b$ with  $u\b$ for some $u\in U(\LL^+)$ if necessary, we
may assume that $v_{0}=\bar{w}$. Choose $n$ such that
$d_jv_{-i}=I_jv_{-i}=0,\forall j\ge n, i=1,2,\ldots,s$. Note that
$A(a,b;F)$ is simple as $\V^{(n)}+\H+\C z_2$ module. Therefore from
Lemma \ref{density}, we may choose some $u\in U(\V^{(n)}+\H+\C z_2)$
with $ux^{k+i}=\delta_{0,i}x^{0}$ for all $j=1,2,\ldots,s.$ Rewrite
$u=\sum_i u_iu_i'$ with $u_i\in U(\H)$ and $u_i'\in U(\V^{(n)})$.
Note that
$$I_jI_i X=FI_{i+j} X,\,\,\forall
i,j\in \Z, X\in A(a,b;F).$$  For sufficient large $l$, replacing
$I_jI_i$ with $FI_{i+j}$ in $I_lu$,  we obtain $u'\in
U(\LL^{(n,n)})$ with $u'x^{k+i}=t^lux^{k+i}=F \delta_{0,i}x^{l},
\forall i=0,1,\ldots,s$. Now $0\ne u'\b=F \bar{w}\otimes x^{l}\in
M$.
\end{proof}

Now we use the ``shifting technique". We will denote $v_k\otimes
x^{i}$ as $v_k\otimes y^{i+k}$ for all $i,k\in \Z$, where $v_k\in
V(\dot{I_0},\dot{d_0},\dot{z_1}, \dot{z_2}, \dot{z_3})$ with $d_0
v_k=(\dot{d_0}+k)v_k$. Then
$$M(\dot{I_0},\dot{d_0},\dot{z_1},\dot{z_2}, \dot{z_3})\otimes
A(a,b;F) =M(\dot{I_0},\dot{d_0},\dot{z_1},\dot{z_2},
\dot{z_3})\otimes \C[y,y^{-1}]$$ with the actions

\begin{equation}\label{shift-1}d_n(u_k w\otimes y^{i})=((d_n-k+a+i+nb)u_k w)\otimes
y^{n+i},\end{equation}

\begin{equation}\label{shift-2}I_n(u_k w\otimes y^i)=((I_n+F)u_k w)\otimes y^{i+n},\end{equation}
for all $u_k\in U(\LL)_k$.
For simplicity we define
 $$W^{(k)}=\sum_{i\in \Z_+}U(\LL)(w\otimes
y^{k+i})\subset M(\dot{I_0},\dot{d_0},\dot{z_1},\dot{z_2},
\dot{z_3})\otimes A(a,b;F),$$
$$W_n^{(k)}=W^{(k)}\cap(
M(\dot{I_0},\dot{d_0},\dot{z_1},\dot{z_2}, \dot{z_3})\otimes y^n),
\forall n\in \Z.$$

\begin{lemma}\label{lemma-20}
\begin{enumerate}\item $W^{(k)}=\sum_{i\in \Z_+} U(\LL^-)(w\otimes y^{k+i})$.
\item  $W^{(k)}\supset \oplus_{i\ge k}
M(\dot{I_0},\dot{d_0},\dot{z_1},\dot{z_2}, \dot{z_3})\otimes y^{i}$.
\item $M(\dot{I_0},\dot{d_0},\dot{z_1},\dot{z_2}, \dot{z_3})\otimes
y^{k-1}=W^{(k)}_{k-1}\oplus \C(w\otimes y^{k-1})$.
\item  Suppose that $P$ is a weight vector in $U(\LL^-)$ such that $$P
w\otimes y^{k-1}\in W^{(k)}_{k-1}\subset
M(\dot{I_0},\dot{d_0},\dot{z_1},\dot{z_2}, \dot{z_3})\otimes
A(a,b;F),$$ then $(U(\LL^-)P w)\otimes y^{k-1} \subset
W^{(k)}_{k-1}$.
\end{enumerate}
\end{lemma}

\begin{proof} (1). It follows from $U(\LL)(w\otimes y^i)=U(\LL^-)U(\LL^++\LL^0)(w\otimes y^i)\subset\sum_{j\in \Z_+} U(\LL^-)(w\otimes
y^{i+j})$.

(2). Using (1), (\ref{shift-1}) and (\ref{shift-2}), by induction on
$s+m$ it is straightforward to prove   that $I_{-j_1}\ldots
I_{-j_s}d_{-l_1}\cdots d_{-l_m} w\otimes y^i\in W^{(k)}$ for all
$i\ge k$ and $j_1,\ldots,j_s,l_1,\ldots,l_m\in \N$.

(3). This follows from (2) and the proof of Lemma
\ref{Verma-nonsimple}.

(4). Suppose that $P\in U(\LL^-)_m$. From (2), (\ref{shift-1}) and
(\ref{shift-2}), we have
\begin{equation}\aligned (d_{-i}P w)\otimes y^{k-1}=&d_{-i}(P w\otimes
y^{k+i-1}))\\ & -(a-m+k-1-ib) (P w)\otimes y^{k-1}\in
W^{(k)}_{k-1},\endaligned\end{equation}

\begin{equation}\aligned (I_{-i}P w)\otimes y^{k-1}=&I_{-i}(P w\otimes
y^{k+i-1}))\\ &-FP w\otimes y^{k-1}\in W^{(k)}_{k-1},\forall i\in
\N.\endaligned\end{equation} Therefore we may prove (4) by induction
on $m$.
 \end{proof}

 For any $n\in \N$, from Lemma \ref{lemma-20} (3), similar to $\phi_n$ defined in \cite{CGZ} we may define the linear map
 $\rho_n:U(\LL^-)\rightarrow \C$
 inductively as follows

\begin{equation}\label{rho-1}\rho_n(1)=1,\end{equation}
 \begin{equation}\label{rho-2}\rho_n(I_{-i}u)=-F\rho_n(u),\end{equation}
 \begin{equation}\label{rho-3}\rho_n(d_{-i}u)=-(a+k+i+n-ib)\rho_n(u), \forall u\in
 U(\LL^-)_{-k}.\end{equation}

\begin{remark} It is clear that  $\rho_n$  depends only on $a,b,F,n$.
\end{remark}

\begin{lemma} \label{24}   Let $P \in
U(\LL_-)$. Then \begin{enumerate}\item $Pw
\otimes t^n \equiv \rho_n(P)w\otimes t^n\,\, (\mod W^{(n+1)});$
\item    $Pw \otimes t^n \in
W^{(n+1)}$ if and only if $\rho_n(P)=0$.\end{enumerate}
\end{lemma}

\begin{proof}  The proof for (1) is similar to that of Lemma 8 in
\cite{CGZ} but one has also to consider $I_{-k}$. Part (2) follows from (1).\end{proof}

Let us recall some results from \cite{B}.

\begin{theorem}\cite{B}\label{billig} Let $(\dot{I_0},\dot{d_0},\dot{z_1},\dot{z_2},\dot{z_3})\in \C^5$ with $\dot{z_3}=0$. \begin{enumerate}
\item If $\frac{\dot{I_0}}{\dot{z_2}}\notin \Z$ or
$\frac{\dot{I_0}}{\dot{z_2}}= 1$, then
$M(\dot{I_0},\dot{d_0},\dot{z_1},\dot{z_2},0)$ is simple.
\item
If $1-\frac{\dot{I_0}}{\dot{z_2}}\in \N$, then
$M(\dot{I_0},\dot{d_0},\dot{z_1},\dot{z_2},0)$ possesses a singular
weight vector $Qw\in (d_{-p}+U(\LL^-)\H^-)w$ and the quotient
module\\ $M(\dot{I_0},\dot{d_0},\dot{z_1},\dot{z_2},0)/U(\LL^-)Qw$
is simple.

\item
If $\frac{\dot{I_0}}{\dot{z_2}}-1\in \N$, then
$M(\dot{I_0},\dot{d_0},\dot{z_1},\dot{z_2},0)$ possesses a singular
weight vector $Qw\in U(\H^-)\H^-$ and the quotient module\\
$M(\dot{I_0},\dot{d_0},\dot{z_1},\dot{z_2},0)/U(\LL^-)Qw$ is
simple.\end{enumerate}
 \end{theorem}

\begin{remark}(1) Let $M(\dot{d_0}, \dot{z_1})$ be the Verma module
 over $\V$. Then it is well-known that (see \cite{FF}, for example), there
 exist two weight vectors $Q_1w', Q_2w'$ such that $U(\V^-) Q_1w'+U(\V^-) Q_2 w'$
 is the maximal proper submodule of $M(\dot{d_0}, \dot{z_1})$, where
$Q_1, Q_2\in U(\V^-)$ may be zero or be equal and $w'$ is the
highest weight vector in $M(\dot{d_0}, \dot{z_1})$.

  (2) If $\dot{z_3}=\dot{z_2}=0$ and $\dot{I_0}\ne 0$,  from \cite{ACKP}
 or Theorem 1 in \cite{CG} we know that
 $M(\dot{I_0},\dot{d_0},\dot{z_1},\dot{z_2},\dot{z_3})$ is simple.

  (3) If $\dot{z_3}\ne 0$, denote by $H$
the highest weight $\H$-module with $I_0=\dot{I_0}$ and
$z_3=\dot{z_3}$. From Theorem \ref{thm-whittaker-1} or \cite{ACKP}
we know that
$$M(\dot{I_0},\dot{d_0},\dot{z_1},\dot{z_2},\dot{z_3})\cong
H(\dot{z_2})\otimes
M(\dot{d_0}+\frac{1}{2\dot{z_3}}\dot{I_0}^2-\frac{\dot{z_2}}{\dot{z_3}}\dot{I_0},\dot{z_1}-1
+\frac{12\dot{z_2}^2}{\dot{z_3}})^\LL.$$ Let
$J=U(\Vir^-)Q_1w_2+U(\Vir^-)Q_2w_2$ be the maximal proper submodule
of
$M(\dot{d_0}+\frac{1}{2\dot{z_3}}\dot{I_0}^2-\frac{\dot{z_2}}{\dot{z_3}}\dot{I_0},\dot{z_1}-1
+\frac{12\dot{z_2}^2}{\dot{z_3}})$ where $Q_1,Q_2\in U(\Vir^-)$ and
$w_2$ is a highest weight vector. From Theorem \ref{thm-whittaker-1} we know
that  $U(\LL^{-})(w_1\otimes Q_1w_2)+U(\LL^-)(w_1\otimes Q_2w_2)$ is
the unique maximal proper submodule of $H(\dot{z_2})\otimes
M(\dot{d_0}+\frac{1}{2\dot{z_3}}\dot{I_0}^2-\frac{\dot{z_2}}{\dot{z_3}}\dot{I_0},\dot{z_1}-1+\frac{12\dot{z_2}^2}{\dot{z_3}})^\LL,$
where $w_1,w_2$ are highest weight vectors of $H(\dot{z_2})$ and
$M(\dot{d_0}+\frac{1}{2\dot{z_3}}\dot{I_0}^2-\frac{\dot{z_2}}{\dot{z_3}}\dot{I_0},\dot{z_1}-1+\frac{12\dot{z_2}^2}{\dot{z_3}})$
respectively. So we have proved that the maximal proper
$\LL$-submodule of
$M(\dot{I_0},\dot{d_0},\dot{z_1},\dot{z_2},\dot{z_3})$ can also be
generated by at most two weight singular vectors.
\end{remark}

Suppose that $(\dot{I_0}, \dot{z_2},\dot{z_3})\ne (0,0,0)$, from the
remark above, we always have weight vectors (or zero)  $ \widetilde
Q_1, \widetilde Q_2\in U(\LL^-)$ such that the maximal proper
$\LL$-submodule of
$M(\dot{I_0},\dot{d_0},\dot{z_1},\dot{z_2},\dot{z_3})$ is generated
by $\widetilde{Q}_1 w,  \widetilde{Q}_2 w$ as $U(\LL^-)$ modules.

\begin{lemma}\label{lemma-24} Suppose that $(\dot{I_0}, \dot{z_2},\dot{z_3})\ne
(0,0,0)$ and
$V(\dot{I_0},\dot{d_0},\dot{z_1},\dot{z_2},\dot{z_3})\otimes
A'(a,b;F)$ satisfies

  {\bf Condition B:} For any nonzero submodule $V_1$ of
$V(\dot{I_0},\dot{d_0},\dot{z_1},\dot{z_2},\dot{z_3})\otimes
A'(a,b,F)$, there exists some $k$ (depending on $V_1$), such that
$\bar{w}\otimes x^{k+i}\in V_1$ for all $i\in \Z_+$.
\begin{enumerate}\item Suppose that $A(a,b;F)$ is simple. Then
$V(\dot{I_0},\dot{d_0},\dot{z_1},\dot{z_2},\dot{z_3})\otimes
A(a,b;F) $ is simple if and only if
$(\rho_n(\widetilde{Q}_1),\rho_n(\widetilde{Q}_2))\ne (0,0)$ for all
$n\in \Z$. \item If $(a,b,F)=(0,0,0)$, then
$V(\dot{I_0},\dot{d_0},\dot{z_1},\dot{z_2},\dot{z_3})\otimes
A'(0,0,0) $ is simple if and only if
$(\rho_n(\widetilde{Q}_1),\rho_n(\widetilde{Q}_2))\ne (0,0)$ for all
$n\in \Z\backslash\{0\}$. \end{enumerate}
\end{lemma}

\begin{proof}Denote by $J=U(\LL^-)\widetilde{Q}_1 w+U(\LL^-)\widetilde{Q}_2
w$ the maximal proper submodule of
$M(\dot{I_0},\dot{d_0},\dot{z_1},\dot{z_2},\dot{z_3})$.

(1) From condition B and Lemmas \ref{lemma-20}, \ref{24}, it is
clear that the module
$V(\dot{I_0},\dot{d_0},\dot{z_1},\dot{z_2},\dot{z_3})\otimes
A(a,b;F)$ is simple if and only if $J\otimes
y^n+W_n^{(n+1)}=M(\dot{I_0},\dot{d_0},\dot{z_1},\dot{z_2},\dot{z_3})\otimes
y^n$ for all $n\in \Z$ if and only if $(U(\LL^-)\widetilde{Q}_1
w+U(\LL^-)\widetilde{Q}_2w)\otimes y^n\not\subset W_n^{(n+1)}$ for
all $n\in \Z$ if and only if $\{Q_1w\otimes y^n,Q_2w\otimes
y^n\}\not\subset W_n^{(n+1)}$ for all $n\in \Z$ if and only if
$(\rho_n(\widetilde{Q}_1),\rho_n(\widetilde{Q}_2))\ne (0,0)$ for all
$n\in \Z$.

(2) The proof is similar to that of Part (1). The only difference is
that we don't need $(\rho_0(Q_1),\rho_0(Q_2))\ne (0,0)$ since the
image of $v\otimes y^0$ is zero in
$V(\dot{I_0},\dot{d_0},\dot{z_1},\dot{z_2},\dot{z_3})\otimes
A'(0,0,0)$.
\end{proof}

Now we summarize the established results into the following main
result in this section.

\begin{theorem} Let $a,b, F,\dot{I_0},\dot{d_0},\dot{z_1},\dot{z_2},\dot{z_3}\in \C$.

 \begin{enumerate}
\item If $F\ne 0$, then $V(0,\dot{d_0},\dot{z_1},0,0)\otimes
A'(a,b;F)$ is simple.

\item  $V(0,\dot{d_0},\dot{z_1},0,0)\otimes
A'(a,b;0)$ is simple if and only if $V(\dot{d_0},\dot{z_1})\otimes
A'(a,b)$
 is a simple $\Vir$-module, which is determined in \cite{CGZ}.

 \item Suppose that $(\dot{I_0}, \dot{z_2},\dot{z_3})\ne (0,0,0)$ and $A(a,b;F)$ is simple. Then
$V(\dot{I_0},\dot{d_0},\dot{z_1},\dot{z_2},\dot{z_3})\otimes
A(a,b;F) $ is simple if and only if
$(\rho_n(\widetilde{Q}_1),\rho_n(\widetilde{Q}_2))\ne (0,0)$ for all
$n\in \Z$.
\item Suppose that $(\dot{I_0}, \dot{z_2},\dot{z_3})\ne (0,0,0)$ and $(a,b,F)=(0,0,0)$. Then
$V(\dot{I_0},\dot{d_0},\dot{z_1},\dot{z_2},\dot{z_3})\otimes
A'(0,0,0) $ is simple if and only if the pairs
$(\rho_n(\widetilde{Q}_1),\rho_n(\widetilde{Q}_2))\ne (0,0)$ for all
$n\in \Z\backslash\{0\}$.\end{enumerate}
\end{theorem}

\begin{proof} (1) is from Lemma \ref{0-0-0}. (2) is trivial. If $\dot{z_3}=\dot{z_2}= 0$
and $\dot{I_0}\ne0$, we have the theorem from Corollary
\ref{simplicity-00}.  If $F=0$ and $\dot{z_3}\ne0$, the theorem
follows from Lemma \ref{0-0-0}.  For other case, we need to check
the Condition B in Lemma \ref{lemma-24}. If $F\ne 0$, we have
Condition B hold from Lemma \ref{lemma-19}. If $F=\dot{z_3}=0$ and
$\dot{z_2}\ne 0$, we have Condition B hold from the proof of Theorem
45 in \cite{R}.
\end{proof}

The simplicity of the tensor product $A'(a,b;0)\otimes
V(\dot{I_0},\dot{d_0},\dot{z_1},\dot{z_2}, 0)$ is obtained in
\cite{R}. Let us recover it as an example.

\begin{theorem}\cite{R}\label{R} Suppose that $\dot{z_3}=0$ and $\dot{z_2}\ne 0$.
 \begin{enumerate}
\item  For $1-\frac{\dot{I_0}}{\dot{z_2}}=p\in \N$, the module
$A'(a,b;0)\otimes V(\dot{I_0},\dot{d_0},\dot{z_1},\dot{z_2}, 0)$ is
simple if and only if $a-pb\notin \Z$.
\item  For $\frac{\dot{I_0}}{\dot{z_2}}-1\in \N$,  the module $A'(a,b;0)\otimes
V(\dot{I_0},\dot{d_0},\dot{z_1},\dot{z_2}, 0)$ is not simple.
\item  The module $A'(a,b;F)\otimes
V(0,0,\dot{z_1},\dot{z_2}, 0)$ is simple if and only if $a-b\notin
\Z$.\end{enumerate}
\end{theorem}

\begin{proof}Suppose that $F=0$. Using Lemma \ref{billig}, from easy computations we
see
that $\rho_n(Q)=0$ if $\frac{\dot{I_0}}{\dot{z_2}}-1\in \N$, and
$\rho_n(Q)=\rho_n(d_{-p})=-(a+p+n-pb)$ if
$1-\frac{\dot{I_0}}{\dot{z_2}}=p\in \N$,  which imply (1) and (2).

Now suppose that
$(\dot{I_0},\dot{d_0},\dot{z_1},\dot{z_2},\dot{z_3})=(0,0,\dot{z_1},\dot{z_2},
0)$, then it is straightforward to see that $Q=d_{-1}$. And we have
$\rho_n(Q)=\rho_n(d_{-1})=-(a+1+n-b)$, which implies (3).\end{proof}

\begin{example}Suppose that $\dot{z_3}\ne 0$ and $F\ne 0$. Let $H$ be the highest weight module over $\H$ with $I_0=\dot{I_0}$ and $z_3=\dot{z_3}$.
Then
$M(\dot{I_0},-\frac{\dot{I_0}^2}{\dot{2z_3}}+\frac{\dot{I_0}\dot{z_2}}{\dot{z_3}},
2-\frac{12\dot{z_2}^2}{\dot{z_3}},\dot{z_2},\dot{z_3})\cong
H(\dot{z_2})\otimes M(0,1)$. The maximal submodule of $M(0,1)$
is generated by the weight vector of weight $-1$. It is
straightforward to obtain that
$\widetilde{Q}_1=\widetilde{Q}_2=d_{-1}+\frac{\dot{I_0}}{\dot{z_3}}I_{-1}$.

We know that
$\rho_n(\widetilde{Q}_1)=-(a+1+n-b)-\frac{\dot{I_0}}{\dot{z_3}}F$.
Therefore
$V(\dot{I_0},-\frac{\dot{I_0}^2}{\dot{2z_3}}+\frac{\dot{I_0}\dot{z_2}}{\dot{z_3}},2-\frac{12\dot{z_2}^2}{\dot{z_3}},\dot{z_2},\dot{z_3})\otimes
A(a,b,F)$ is simple if and only if
$b-a-\frac{\dot{I_0}}{\dot{z_3}}F\notin \Z$. If
$n=b-a-1-\frac{\dot{I_0}}{\dot{z_3}}F\in \Z$. So $U(\LL)\cdot
(\bar{w}\otimes y^{n+1})$ is the unique minimal submodule of
$V(\dot{I_0},-\frac{\dot{I_0}^2}{\dot{2z_3}}+\frac{\dot{I_0}\dot{z_2}}{\dot{z_3}},2-\frac{12\dot{z_2}^2}{\dot{z_3}},\dot{z_2},\dot{z_3})\otimes
A(a,b;F)$, which is simple and the quotient
$$V(\dot{I_0},-\frac{\dot{I_0}^2}{\dot{2z_3}}+\frac{\dot{I_0}\dot{z_2}}{\dot{z_3}},2-\frac{12\dot{z_2}^2}{\dot{z_3}},\dot{z_2},\dot{z_3})\otimes
A(a,b;F)/(U(\LL)\cdot (\bar{w}\otimes y^{n+1}))$$ is a highest weight
module with the highest weight vector $\bar{w}\otimes y^n+U(\LL)\cdot
(\bar{w}\otimes y^{n+1})$.

\end{example}

\section{Simple $\LL$-modules from $\mathcal{O}_{\WW}$}

We will use the algebras defined in Sect.1: $\WW$,  $\aa$, $\LL^{(r,s)}$ and $\LL[\lambda]$.

In this section we will first classify all simple modules in
$\mathcal{O}_{\WW}$, then use the ``embedding trick" to make these
simple $\WW$-modules into simple $\LL$-modules.

\subsection{Simple modules in $\mathcal{O}_{\WW}$}

For any  $B\in \mathcal{O}_\aa$ and $0\ne v\in B$, define
$\ord_{\aa}(v)$, the order of $v$, to be  the minimal nonnegative
integer $r$ with $I_{r+i} v=0$ for all $i\ge 0$. And
$\ord_{\aa}(B)$, the order of $B$, is defined to be the maximal
order of all its elements or $\infty$ if it doesn't exist.

\begin{lemma}\label{char-O(b)}Suppose that $B\in \mathcal{O}_{\aa}$ is
simple.\begin{enumerate}
 \item We have
$\ord_{\aa}(B)=\ord_{\aa}(v)$ for all $ 0\ne v\in B$. And there
exists some $(r,s)\in \Z_+^2$ such that $\LL^{(r,s)} B=0$, i.e., $B$
can be regarded as a simple module over $\aa_{r,s}$.
\item If $\ord_{\aa}(B)=0$, then $B$ is a simple
$\ma$-module.
\item If $r=\ord_{\aa}(B)>0$, then the action
of $I_{r-1}$ on $B$ is bijective.\end{enumerate}\end{lemma}

\begin{proof}For any nonzero $v,v'\in B$, since $B$ is simple, there exists some $u\in U(\aa)$, such that $v'=uv$.
It is straightforward to check that $I_i v'=I_iuv=0, \forall i\ge
\ord_{\aa}(v)$. So $\ord_{\aa}(v)\ge \ord_{\aa}(v')$. Similarly we
have $\ord_{\aa}(v')\ge \ord_{\aa}(v)$. Thus
$\ord_{\aa}(v')=\ord_{\aa}(v)$.  Now suppose that $d_i v=0,\forall
i\ge k$. Then take $s=\max\{k,r\}$. It is easy to verify that
$\LL^{(r,s)}B=0$. So we have proved (1). Part (2) is trivial.

Now suppose that $r=\ord_{\aa}(B)>0$.
Consider the subspace $X=\{v\in B|I_{r-1} v=0\}$ which is a proper
subspace of $B$. Then $X$ and $I_{r-1} B$ are $\aa$-submodules of
$B$. Since $B$ is simple, $X\ne B$ and $I_{r-1} B\ne 0$, we deduce
that $X=0$ and $I_{r-1} B=B$, i.e., $I_{r-1} B$ is bijective. Part
(3) follows.
\end{proof}

\begin{lemma}\label{char-O(ww)}Let $B\in
\mathcal{O}_{\aa}$, $W,W_1\in\mathcal{O}_{\WW}$ be nontrivial simple
modules. \begin{enumerate} \item The module $\Ind_{\aa}^{\WW}(B)$ is simple
 in $\mathcal{O}_{\WW}$;
\item The module  $\Soc_{\aa}(W)$ is a simple $\aa$-module, and
an essential
$\aa$-submodule of $W$;
\item We have $W\cong\Ind_{\aa}^{\WW}\Soc_{\aa}(W),
B= \Soc_{\aa}(\Ind_{\aa}^{\WW}B)$;
\item We have $W\cong W_1$ if and only if $\Soc_{\aa}(W)\cong
\Soc_{\aa}(W_1)$.\end{enumerate}
\end{lemma}

\begin{proof}(1). Let $M$ be any nonzero submodule of $\Ind_{\aa}^{\mathfrak{\WW}}(B)=\C[d_{-1}]\otimes B$.
Choose $0\ne v=\sum_{i=0}^s d_{-1}^i\otimes v_i\in M$ with minimal
$s$, where $v_i\in B$. Denote $r=\ord_{\aa}(B)$. If $r=0$, the
result follows from Lemma \ref{char-O(w)}. Thus we assume that
$r>0$. If $s>0$, then $$0\ne I_{r} v\in -srd_{-1}^{s-1}\otimes
I_{r-1} v_s+\sum_{i=0}^{s-2} d_{-1}^i \otimes B\subset M,$$ which
contradicts to the minimality of $s$. So $s=0$, i.e., $v\in 1\otimes
B$. Therefore $M=\Ind_{\aa}^{\WW}(B)$, and $\Ind_{\aa}^{\WW}(B)$ is
simple.

(2). Fix some $0\ne w\in W$ with minimal $\ord_{\aa} w=r$. Let
$M=U({\aa})w$. Then $\ord_{\aa}M=r$, and
$W=\C[d_{-1}] M$. If $r=0$, the result follows from Lemma \ref{char-O(w)}. Thus we assume that $r>0$. 
For any $v=\sum_{i=0}^s d_{-1}^i w_i\in W$ with $w_s\ne 0$ and
$w_i\in M$ for $i=0,\ldots,s$, we have
 \begin{equation}\label{essential}0\ne I_{r+s-1} v=(-1)^s(r+s-1)(r+s)\cdots r
 I_{r-1}w_s\in M,\end{equation} \begin{equation}I_{r+i} v=0,\forall
i\ge s.\end{equation} Thus $\ord_{\aa}(v)=r+s$. So
$W=\C[d_{-1}]\otimes M$ and $W\cong \Ind_{\aa}^{\WW}M$, thus $M$ is
simple as $\aa$-module, and it is essential from (\ref{essential}).

Part (3) is an obvious consequence of (1) and (2). Part (4) follows
from (3).
\end{proof}
 \begin{example}\label{exmp-kappa} Consider some $r\in\N$ and set $$\mu=(\mu_{r},\mu_{r+1},\dots,
\mu_{2r}),\kappa=(\kappa_0, \kappa_1,\kappa_2,\ldots,\kappa_r)\in
\C^{r+1}.$$  Define the one dimensional
 $\LL^{(0,r)}$-module $\C v_{\mu,\kappa}$ with the action
 $$d_{2r+i} v_{\mu,\kappa}=I_{i+r} v_{\mu,\kappa}=0,\,\,\forall i\in \N,$$
 $$I_i v_{\mu,\kappa}=\kappa_iv_{\mu,\kappa}, \,\,\,d_{r+i} v_{\mu,\kappa}=\mu_{r+i}v_{\mu,\kappa},\,\,\forall k=0,1,\ldots, r.$$
 Then we have the induced $\WW$-module $W_{\mu,\kappa}=\Ind_{\LL^{(0,r)}}^{\WW}
 \C v_{\mu,\kappa}$.\end{example}

 \begin{lemma} The $\WW$-module   $W_{\mu,\kappa}$ is simple if and only if $(\mu_{2r},\mu_{2r-1},\kappa_r)\ne (0,0,0)$.
\end{lemma}
\begin{proof}If $(\mu_{2r},\mu_{2r-1})\ne (0,0)$, then from [MZ] or [LGZ], we know that $W_{\mu,\kappa}$ is a simple $\W$-module, hence a simple $\WW$-module.
Now suppose that $(\mu_{2r},\mu_{2r-1})=(0,0)$ and $\kappa_r\ne 0$. We
make $\C v_{\mu,\kappa}$ to be a $\LL^{(0,r)}+\sum_{i=1}^3\C z_i$
module by $z_1 =z_2=0, z_3=1$. Then by Theorem
\ref{thm-whittaker-1}, $\Ind_{\LL^{(0,r)}+\sum_{i=1}^3\C
z_i}^{\LL}\C v_{\mu,\kappa}$ is a simple $\LL$-module. Thus
$\Ind_{\LL^{(0,r)}+\sum_{i=1}^3\C z_i}^{\WW+\sum_{i=1}^3 \C z_i}\C
v_{\mu,\kappa}$ is a simple $\WW+\sum_{i=1}^3 \C z_i$ module, that
is, $W_{\mu,\kappa}$ is a simple $\WW$-module. On the other hand, if
$(\mu_{2r},\mu_{2r-1},\kappa_r)=(0,0,0)$, then it is straightforward
to verify that $d_{r-1}v_{\mu,\kappa}$ generates a proper $\WW$
submodule.
 \end{proof}

\subsection{The ``embedding trick"}
Let $W\in \mathcal{O}_{\WW}$. Then $W$
 can be naturally regarded as a  module over $\C[[t]]\frac{d}{d t}+\C[[t]]$.
 Regard $\C[t,(t+\lambda)^{-1}]\frac{d}{ d t}+\C[t,(t+\lambda)^{-1}]$ as a subalgebra of $\C[[t]]\frac{d}{d t}+\C[[t]]$. We will use the expression
 $$(t+\lambda)^{m}=\sum_{i=0}^{\infty}\binom{m}{i}\lambda^{m-i}t^i
 \in\C[[t]], \,\forall\, m\in\Z,\lambda\in \C^*,$$ where $\binom{m}{i} =\frac{m\cdot (m-1)\cdots
                                        (m-i+1)}{i!}.$
Let $$\sigma_{\lambda}:\LL \rightarrow
\C[t,(t+\lambda)^{-1}]\frac{d}{ d t}+\C[t,(t+\lambda)^{-1}]$$ be the
epimorphism of Lie algebras defined by
$\sigma_{\lambda}(f(t)\frac{d}{dt})=f(t+\lambda)\frac{d}{dt}$ and
$\sigma_{\lambda}(z_i)=0,i=1,2,3$. So we have the $\LL$-module
$W[\lambda]=W$ with the action
\begin{equation}z_i\circ w=0, i=1,2,3,\end{equation}
\begin{equation}(f(t)\frac{d}{d t})\circ w =\sigma_{\lambda}(f(t)\frac{d}{d t})v=(f(t+\lambda)\frac{d}{dt}) v,
\end{equation}
\begin{equation}f(t)\circ w =\sigma_{\lambda}(f(t))v=f(t+\lambda) v, \forall\,\,
f(t)\in \C[t,t^{-1}], w\in W.\end{equation} We call the above method
to make $\WW$-module $W$ into $\LL$-module $W[\lambda]$ the
``embedding trick".

\begin{remark}\label{WW equivalent}Note that $\sigma_{\lambda}|_{\WW}$ is a Lie algebra automorphism. Hence for any $W\in \mathcal{O}_{\WW}$, $W[\lambda]$ is equivalent to $W$ as $\WW$-modules. It is easy to see that $W[\lambda]\notin \mathcal{O}_{\LL}$ for any $\lambda\in\C^*$ unless $W$ is trivial.\end{remark}

\begin{proposition}\label{thm10}Suppose that $W, W'\in \mathcal{O}_{\WW}$
are simple and nontrivial,  $\lambda,\lambda'\in \C^*$.
\begin{enumerate}[$($a$)$] \item The module $W[\lambda]$ is a simple $\LL$-module. \item We have $W[\lambda]\cong
\Ind_{\LL[\lambda]}^{\LL}\Soc_{\aa}(W)= \C[d_0]\otimes
\Soc_{\aa}(W)$, where $\Soc_{\aa}(W)$ is considered as a
${\LL[\lambda]}$-module.
\item We have $W[\lambda]\cong W'[\lambda']$ as $\LL$-modules if and only if
$\lambda=\lambda'$ and $W\cong W'$ as $\WW$-modules.
\end{enumerate}
\end{proposition}

\begin{proof} Part (a) follows from Remark \ref{WW equivalent}.

(b). For any $w\in \Soc_{\aa}(W)$, any $n\in\Z$, we have
$$(t^n)\circ w=(t+\lambda)^nw=\sum_{i=0}^{\infty}\binom ni \lambda^{m-i}t^{i}w,$$
$$((t-\lambda)t^n\frac d{dt})\circ w=(t(t+\lambda)^n\frac d{dt})w=\sum_{i=1}^{\infty}\binom ni \lambda^{m-i}(t^{i+1}\frac d{dt})w.$$
So  $\Soc_{\aa}(W)$ is a $\LL[\lambda]$-module. The rest
follows from the definition of $W[\lambda]$ and Lemma \ref{char-O(ww)}.

The sufficiency of Part (c) is trivial. Now suppose that
$\psi:W[\lambda]\rightarrow W'[\lambda']$ is a $\WW$-module
isomorphism. Suppose that $\lambda\ne \lambda'$.
 For any $0\ne v\in W[\lambda]$, there exists some $k\in \N$ such that
 $(\C[t,t^{-1}](t-\lambda)^k\frac{d}{d t}+\C[t,t^{-1}](t-\lambda)^k+\C z_1+\C z_2+\C z_2 )\circ v=0.$
 Hence \begin{equation}\label{e1}(\C[t,t^{-1}](t-\lambda)^k\frac{d}{d t}+\C[t,t^{-1}](t-\lambda)^k+\sum_{i=1}^3\C z_i)\circ\psi(v)=0.\end{equation} And also
 there exists some $k'\in \N$, such that
 \begin{equation}\label{e2}(\C[t,t^{-1}](t-\lambda')^{k'}\frac{d}{d t}+\C[t,t^{-1}](t-\lambda')^{k'}+\sum_{i=1}^3\C z_i )\circ\psi(v)=0.\end{equation}
 From (\ref{e1}) and (\ref{e2}), we have $\LL \circ \psi(v)=0$, a contradiction. So $\lambda=\lambda'$.
 For any $f(t)\frac{d}{d t}+g(t)\in \mathcal{W}$ and $v\in W$, we have
 $$\psi((f(t)\frac{d}{d t}+g(t)) v)=\psi((f(t-\lambda)\frac{d}{d t}+g(t-\lambda))\circ v)$$ $$=(f(t-\lambda)\frac{d}{d t}+g(t-\lambda))\circ \psi(v)=(f(t)\frac{d}{d t}+g(t))  \psi(v)\in W'.$$ Thus $\psi:W\rightarrow W'$ is a $\mathcal{W}$-module isomorphism. \end{proof}

\begin{example}\label{examp-Omega}Let $\mathcal{A}(b_1,b_2)$ be the
Verma module over $\WW$ with the highest weight vector $v_{b_1,b_2}$
of highest weight $(b_1,b_2)\in\C^2$, i.e., $d_0 v_{b_1,b_2}=b_1
v_{b_1,b_2}$, $I_0v_{b_1,b_2}=b_2v_{b_1,b_2}$, and
$I_iv_{b_1,b_2}=d_iv_{b_1,b_2}=0$ for all $i>0$. Denote
$\Omega(\lambda;b_1,b_2)=\mathcal{A}(b_1,b_2)[\lambda]$. Then
$\Omega(\lambda;b_1,b_2)$ is simple if and only if
$(b_1,b_2)\ne(0,0)$.

  For any $\lambda\in C^*$, the action of $\LL$ on $\Omega(\lambda;b_1,b_2)=\C[d_0]\circ v_{b_1,b_2}$ is
$$\aligned z_1&=z_2=z_3=0,\\ d_m \circ (d_0^i\circ  v_{b_1,b_2})=&(d_0-m)^i\circ  (d_m\circ v_{b_1,b_2})\\
=&(d_0-m)^i\circ  (((t+\lambda)^{m+1}\frac{d}{d t}) v_{b_1,b_2})\\
=&(d_0-m)^i\circ  (\lambda^{m+1}d_{-1}+(m+1)\lambda^{m}d_0) v_{b_1,b_2})\\
=&(d_0-m)^i\circ  ((m\lambda^m d_{0} +\lambda^{m}(d_0+\lambda d_{-1}) v_{b_1,b_2})\\
=&\lambda^m((d_0-m)^i(mb_1+d_{0}))\circ  v_{b_1,b_2},\,\,\forall
f(t)\in \C[t],\endaligned$$
$$\aligned I_m \circ  (d_0^i\circ v_{b_1,b_2})=&(d_0-m)^i\circ (I_m\circ  v_{b_1,b_2})\\
=&(d_0-m)^i\circ (((t+\lambda)^{m+1}) v_{b_1,b_2})\\
=&(\lambda^{m+1}b_2)(d_0-m)^i\circ  v_{b_1,b_2},\forall i\in
\Z.\endaligned$$
We see that,
as $\Vir$-modules, $\Omega(\lambda;b_1,b_2)$ is isomorphic to
$\Omega(\lambda, b_1+1)$ defined in \cite{LZ2}.\end{example}

 \begin{example}\label{examp-1}  Let $W_{\mu,\kappa}$ be as defined in Example \ref{exmp-kappa} with $r=1$.
 Then $W_{\mu,\kappa}[\lambda]=\C[d_0,d_{-1}] v_{\mu,\kappa}$
 is simple if and only if $\mu\ne 0$ or $\kappa_1\ne 0$.
  Take $$\{d_0^i\circ (d_0^jv_{\mu,\kappa})|i,j\in \Z_+\}$$ as a basis of $W_{\mu,\kappa}[\lambda]$.  Then action
 of $\LL$ on this basis is

 $$\aligned z_1&=z_2=z_3=0,\\ &d_m \circ (d_0^i\circ (d_0^jv_{\mu,\kappa}))\\
=&\lambda^{m}((d_0-m)^id_0)\circ (
d_0^jv_{\mu,\kappa})+m\lambda^{m}(d_0-m)^i\circ
(d_0^{j+1}v_{\mu,\kappa})\\
&+\frac{m^2+m}{2}\lambda^{m-1}\mu_1(d_0-m)^i\circ
((d_0-1)^jv_{\mu,\kappa})\\
&+\frac{m^3-m}{6}\lambda^{m-2}\mu_2(d_0-m)^i\circ
(d_0-2)^jv_{\mu,\kappa}),\endaligned$$

 $$\aligned &I_m \circ (d_0^i\circ (d_0^jv_{\mu,\kappa}))\\
 =&(d_0-m)^i
\circ
((\lambda^m\kappa_0d_0^j+m\lambda^{m-1}(d_0-1)^j\kappa_1)v_{\mu,\kappa}).\endaligned$$

\end{example}
\subsection{Simplicity of tensor product modules}
In this subsection we will use Theorem 7 to construct more simple
$\LL$-modules by taking tensor product of simple modules constructed
in this paper.

\begin{lemma}\label{lemma-13} Let $W_1,W_2,\ldots, W_n\in \mathcal{O}_{\WW}$ be simple and
nontrivial, and $\lambda_1,\lambda_2,\ldots,\lambda_n\in \C^*$
pairwise distinct. Then the $\LL$-module
$$W_1[\lambda_1]\otimes W_2[\lambda_2]\otimes\cdots\otimes
W_n[\lambda_n]$$ is simple.
\end{lemma}

\begin{proof}We will prove the lemma by induction on $n$. It is obvious for $n=1$. Now suppose that $n>1$.
Denote $V_1= W_1[\lambda_1]\otimes
W_2[\lambda_2]\otimes\cdots\otimes W_{n-1}[\lambda_{n-1}]$ and
$V_2=W_n[\lambda_n]$.

From the inductive hypothesis, $V_1$ is simple.

Take $p(t)=(t-\lambda_1)\ldots (t-\lambda_{n-1})$. From the
definition of $W[\lambda]$, we see that for any finite subset
$v_1\in V_1$, and $S_2\subset V_2$, there exists some $k_0\in \N$
such that
$$(\C[t,t^{-1}]p(t)^{k_0}\frac{d}{d
t}+\C[t,t^{-1}]p(t)^{k_0}+\C z_1+\C z_2+\C z_2 )\circ v_1=0,$$
$$(\C[t,t^{-1}](t-\lambda_n)^{k_0}\frac{d}{d
t}+\C[t,t^{-1}](t-\lambda_n)^{k_0}+\C z_1+\C z_2+\C z_2 )\circ
S_2=0.$$
 Note that
$\C[t,t^{-1}]p(t)^{k_0}+\C[t,t^{-1}](t-\lambda_n)^{k_0}=\C[t,t^{-1}].$
From Theorem \ref{thm1}, we know that $V_1\otimes
V_2=W_1[\lambda_1]\otimes W_2[\lambda_2]\otimes\cdots\otimes
W_n[\lambda_n]$ is simple.
\end{proof}

\begin{theorem}\label{nonweight-tensor} Let $W_1,W_2,\ldots, W_n\in \mathcal{O}_{\WW}$
be simple and nontrivial, $V\in \mathcal{O}_{\LL}$ be
simple, and $\lambda_1,\lambda_2,\ldots,\lambda_n\in \C^*$ be
pairwise distinct. Then the $\LL$-module
$$V\otimes W_1[\lambda_1]\otimes W_2[\lambda_2]\otimes\cdots\otimes
W_n[\lambda_n]$$ is simple.\end{theorem}

\begin{proof} Denote $V_1=V$ and $V_2=W_1[\lambda_1]\otimes
W_2[\lambda_2]\otimes\cdots\otimes W_n[\lambda_n]$. From Lemma
\ref{lemma-13}, $V_2$ is simple. Take $p(x)=(t-\lambda_1)\ldots
(t-\lambda_{n})$. Then see that for any finite subsets $v \subset
V_1$, and $S_2\subset V_2$, there exists some $k_0>0$ such that

$$(\C[t]t^{k_0}\frac{d}{d t}+\C[t]t^{k_0}\frac{d}{d t}) v=0,$$
$$(\C[t,t^{-1}]p(t)^{k_0}\frac{d}{d
t}+\C[t,t^{-1}]p(t)^{k_0}+\C z_1+\C z_2+\C z_2 )\circ S_2=0.$$
Clearly, $(\C[t]t^{k_0}+(\C[t,t^{-1}]p(t)^{k_0}=\C[t,t^{-1}].$ By
Theorem \ref{thm1}, we know $V_1\otimes V_2= V\otimes
W_1[\lambda_1]\otimes W_2[\lambda_2]\otimes\cdots\otimes
W_n[\lambda_n]$ is simple.
\end{proof}

\begin{example} Let  $\lambda_1,\ldots,
\lambda_n\in \C^*$ are pairwise distinct, $a_1,\ldots, a_n,
b_1,$ $\ldots, b_n \in\C$, and $V\in \mathcal{O}_{\LL}$ is simple. From
Example \ref{examp-Omega} and using Theorem \ref{nonweight-tensor},
we obtain the simple module $\Omega(\lambda_1;a_1,b_1)\otimes
\Omega(\lambda_2;a_1,b_2)\otimes\cdots\otimes \Omega(\lambda_n;a_n,
b_n)\otimes V$ if each $(a_i,b_i)\ne (0,0)$. Further, if we take
$b_i=0$ for all $i$, and $V\in \mathcal{O}_{\V}$, then
$\Omega(\lambda_1;a_1,0)\otimes
\Omega(\lambda_2;a_1,0)\otimes\cdots\otimes \Omega(\lambda_n;a_n,
0)\otimes V$ is also a simple $\V$-module since $\H$ acts trivially.
Such a simple $\Vir$-module is obtained in [TZ1, TZ2].
\end{example}

\subsection{Isomorphism classes}
In this subsection we will determine the isomorphism classes of the simple tensor product $\LL$-modules discussed in Theorem 34.

\begin{theorem}Suppose that $W_1[\lambda_1]\otimes
W_2[\lambda_2]\otimes\cdots\otimes W_n[\lambda_n]\otimes V$ and
$W_1'[\lambda_1']\otimes W_2'[\lambda_2']\otimes\cdots\otimes
W_m'[\lambda_m']\otimes V'$ are isomorphic simple $\LL$-modules as
in Theorem \ref{nonweight-tensor}. Then $m=n$, $V\cong V'$,
$\lambda_i=\lambda_i'$, $W_i\cong W_i'$ for $i=1,2,\ldots,n$ after
re-indexing the modules $W_m'[\lambda_m']$.
\end{theorem}

\begin{proof} For any $v\in V$ or $V'$, there is an $l\in\N$
such that $$\left(\C[t]t^l\frac{d}{d t}+\C[t]t^l \right)v=0.$$ Let
$p_1(t)=(t-\lambda_1)\ldots(t-\lambda_n)$ and
$p_2(t)=(t-\lambda_1')\ldots(t-\lambda_n').$ From the given
isomorphism, we know that,  for any nonzero $$ X= w_1\otimes \cdots
\otimes w_n\otimes v\in W_1[\lambda_1]\otimes
W_2[\lambda_2]\otimes\cdots\otimes W_n[\lambda_n]\otimes V,$$ there
exists $k_0\in\N$ such that \begin{equation}\left(
\C[t](p_1(t)t)^{k_0}\frac{d}{d t}+\C[t](p_1(t)t)^{k_0} \right)\circ
w_i=0,i=1,2,\ldots,n \end{equation}
\begin{equation}\left( \C[t](p_1(t)t)^{k_0}\frac{d}{d
t}+\C[t](p_1(t)t)^{k_0} \right)v=0, {\text{ and
}}\end{equation}\begin{equation}\label{p2} \left(
\C[t](p_2(t)t)^{k_0}\frac{d}{d t}+\C[t](p_2(t)t)^{k_0}
\right)X=0.\end{equation} If $p_1(t)\nmid p_2(t)$, without lose of
generality, we may assume that $(t-\lambda_1) \nmid p_2(t)$. Let
$p(t)=p_1(t)p_2(t)/(t-\lambda_1)$. Then
 $$\left(\C[t](p(t)t)^{k_0}\frac{d}{d t}+\C[t](p(t)t)^{k_0}\right)(w_2\otimes\cdots \otimes
w_n\otimes v)=0.$$ Combining with (\ref{p2}), we know that
\begin{equation}\label{annw1}\left(\C[t](p(t)t)^{k_0}\frac{d}{d
t}+\C[t](p(t)t)^{k_0}\right)\circ w_1=0.\end{equation}
 Note that there exists some $k\in \Z$ such that
\begin{equation}\label{annw2}\left(\C[t,t^{-1}](t-\lambda_1)^k\frac{d}{d
t}+\C[t,t^{-1}](t-\lambda_1)^k+\sum_{i=1}^3 \C z_i\right)\circ
w_1=0.\end{equation}
 From (\ref{annw1}) and (\ref{annw2}), we have $\LL\circ w_1=0$, which is a contradiction.
 Thus $p_1(t)|p_2(t)$, and
similarly $p_2(t)|p_1(t)$, to give $p_1(t)=p_2(t)$. So we have
$m=n$, and we may assume that $\lambda_i=\lambda_i'$ for
$i=1,2,\ldots,n$.

 Let $\psi: W_1[\lambda_1]\otimes
 \cdots\otimes W_n[\lambda_n]\otimes V\rightarrow
W_1'[\lambda_1]\otimes \cdots\otimes W_n'[\lambda_n]\otimes V'$ be a
$\LL$-module isomorphism. Fix some nonzero $$ X=w_1\otimes \ldots
w_n\otimes v\in W_1[\lambda_1]\otimes
W_2[\lambda_2]\otimes\cdots\otimes W_n[\lambda_n]\otimes V.$$ Write
$\psi(X)=\sum_{i=1}^s w'_{1,i}\otimes Y_i$ with minimal $s$, where
$w'_{1,i}\in W_1'[\lambda_1']$, $Y_i\in
W_2'[\lambda_2]\otimes\cdots\otimes W_n'[\lambda_n]\otimes V'$. We
know that $Y_i$'s are linearly independent.  Denote
$p(t)=t(t-\lambda_2)\ldots(t-\lambda_n)$, then there exists
$k_1\in\N$ such that
$$\left(\C[t]p(t)^{k_1}\frac{d}{d t}+\C[t]p(t)^{k_1}\right)
Y_i=0,\,\, \forall i=1,2,\ldots, s,$$
$$\left(\C[t]p(t)^{k_1}\frac{d}{d t}+\C[t]p(t)^{k_1} \right)(w_2\otimes\cdots\otimes
w_n\otimes v)=0.$$

Denote $L=\C[t]p(t)^{k_1}\frac{d}{d t}+\C[t]p(t)^{k_1}$. Then for
any $w\in W_1[\lambda],w'\in W_1'[\lambda]$, it is straightforward
to check that
\begin{equation}\label{annw_1}\ann_{\LL}(w)+L=\ann_{\LL}(w')+L=\LL.\end{equation}
For any $w\in W_1$, we have $W_1[\lambda_1]=U(\LL)\circ
w=U(\ann_L(w)+L)\circ w=U(L)U(\ann_L(w))\circ w=U(L)\circ w$,
that is, $W_1[\lambda_1]$ is a simple $L$-module. Similarly we have
$W_1'[\lambda_1]$ is also simple as $L$-module. Now from Lemma
\ref{density}, there exists some $u_0\in U(L)$ such that $u_0\circ
w'_{1,i}=\delta_{i,1}w'_{1,1}$. Then
$$\psi((u_0\circ w_1)\otimes (w_2\otimes\cdots w_n\otimes v))=w'_{1,1}\otimes
Y_1.$$ And

$$\psi((uu_0)\circ w_1)\otimes (w_2\otimes\cdots w_n\otimes v)))=(u\circ
w'_{1,1})\otimes Y_1, \forall u\in U(L).$$

Recall that $W_1[\lambda_1], W'_1[\lambda_1]$ are simple $\LL$
-module. Therefore we have the well-defined linear map
$\phi:W_1[\lambda_1]\rightarrow W_2[\lambda_2]$ defined by
$\phi(u\circ (u_0\circ w_1))=u\circ w'_{1,1}$ for all $u\in U(L)$,
which is obvious a $L$-module isomorphism. Now we only need to show
that it is also a $\LL$-module homomorphism. Now for any $w\in
W_1[\lambda]$, it is easy to check that $L+(\ann_{\LL}(w)\cap
\ann_{\LL}\phi(w))=\LL$. Then for any $x\in \LL$, we may write
$x=x_1+x_2$ with $x_1\in L$ and $x_2\in \ann_{\LL}(w)\cap
\ann_{\LL}\phi(w)$. Now $$\phi(x\circ w)=\phi(x_1\circ w)=x_1\circ
\phi(w)=x\circ \phi(w),$$ i.e., $\phi$ is a $\LL$-module isomorphism.

Similarly, we have $W_i[\lambda_i]\cong W_i'[\lambda_i]$ for $
i=2,3,\ldots,n$, and $V\cong V'$. Now the theorem follows from
Proposition \ref{thm10} (c).
 \end{proof}

It is easy to see that the simple modules  for $m\ge 1$ in Theorem
34 are not isomorphic to any generalized oscillator representations.

\begin{center}
\bf Acknowledgments
\end{center}

\noindent K.Z. is partially supported by  NSF of China (Grant
11271109) and NSERC. R.L. is partially suppported by NSF of China (Grant 11371134) and Jiangsu Government Scholarship for Overseas Studies (JS-2013-313).

\vspace{10pt}

\noindent  R.L.: Department of Mathematics, Soochow university,
Suzhou 215006, Jiangsu, P. R. China.
 Email: rencail@amss.ac.cn

\vspace{0.2cm} \noindent K.Z.: Department of Mathematics, Wilfrid
Laurier University, Waterloo, ON, Canada N2L 3C5,  and College of
Mathematics and Information Science, Hebei Normal (Teachers)
University, Shijiazhuang, Hebei, 050016 P. R. China. Email:
kzhao@wlu.ca

\end{document}